\documentclass[twoside]{amsart}
\usepackage{amsmath,amssymb} 
\usepackage[colorlinks=true,
           urlcolor=black,
           ]{hyperref} 
\usepackage{graphicx}
\usepackage[utf8]{inputenc}
\usepackage{verbatim} 
\usepackage{color}
\usepackage[numbers,longnamesfirst]{natbib}

\newcommand{\Ht}{L^2(\Omega, L^2 (\Gamma(t)))}

\newcommand{\Vt}{L^2(\Omega, H^1 (\Gamma(t)))}

\newcommand{\Gt}{\Gamma(t)}
\newcommand{\Ght}{ \Gamma_h(t)}
\newcommand{\Ng}{\nabla_\Gamma}
\newcommand{\Ngh}{\nabla_{\Gamma_h}}
\newcommand{\Rh}{\mathcal R}
\newcommand{\R}{\mathbb{R}}

\newcommand{\mdh}{\partial^\bullet_h}
\newcommand{\md}{\partial^\bullet}
\newcommand{\vfi}{\varphi}
\newcommand{\intt}{\int_0^t}
\newcommand{\w}{\mathrm{v}}

\newcommand\restr[2]{{
  \left.\kern-\nulldelimiterspace 
  #1 
  \vphantom{\big|} 
  \right|_{#2} 
  }}

\newtheorem{lemma}{Lemma}[section]
\newtheorem{proposition}{Proposition}[section]
\newtheorem{theorem}{Theorem}[section]

\newtheorem{xrem}{Remark}[section]
\newtheorem{ass}{Assumption}[section]
\newtheorem{problem}{Problem}[section]

\numberwithin{equation}{section}

\newcounter{todocounter}
\newlength{\todowidthinner}
\usepackage{xifthen}
\usepackage{marginnote}
\DeclareRobustCommand{\MyChange}[3][\empty]{%
  {\color{#2}#3}
  \ifthenelse{\isempty{#1}}{}
  {%
    \addtocounter{todocounter}{1}%
    \ifmmode%
        {\color{#2}\text{$^{\framebox{\arabic{todocounter}}}$}}%
    \else%
        {\color{#2}\text{$^{\arabic{todocounter}}$}}%
    \fi%
    \marginpar{\textcolor{#2}{$^{\arabic{todocounter}}$\textnormal{#1}}}%
  }%
}
\definecolor{lightgray}{rgb}{0.7,0.7,0.7}
\usepackage[normalem]{ulem}

\definecolor{adcol}{rgb}{0,0.7,0.7}

\definecolor{cecol}{rgb}{0,0.7,0}

\definecolor{revisioncol}{rgb}{1,0,0}

\begin{document}

\title[ESFEM for  random advection-diffusion equations]
{Evolving surface finite element methods for  random advection-diffusion equations}

\author[A. Djurdjevac]{Ana Djurdjevac}
\address{Institut f\"ur Mathematik, 
 Freie Universit\"at Berlin, 
 14195 Berlin, Germany}
\email{ana.djurdjevac@fu-berlin.de}

\author[C. M. Elliott]{Charles M. Elliott}
\address{Mathematics Institute,
  University of Warwick,
  Coventry.
  CV4 7AL.
  UK}
\email{C.M.Elliott@warwick.ac.uk}

\author[R. Kornhuber]{Ralf Kornhuber}
\address{Institut f\"ur Mathematik, 
 Freie Universit\"at Berlin, 
 14195 Berlin, Germany}
\email{ralf.kornhuber@fu-berlin.de}

\author[T. Ranner]{Thomas Ranner}
\address{School of Computing,
  University of Leeds,
  Leeds. LS2 9JT. UK}
\email{T.Ranner@leeds.ac.uk}

\thanks{The research of TR was funded by the Engineering and Physical Sciences Research Council (EPSRC EP/J004057/1). 
The research of CME was partially supported by the Royal Society via a Wolfson
Research Merit Award and  by the
EPSRC programme grant (EP/K034154/1) EQUIP}


%
%
%
%


\begin{abstract}
In this paper, we introduce and analyse a surface finite element discretization of  advection-diffusion equations
with uncertain coefficients on evolving hypersurfaces.
After stating unique solvability of the resulting semi-discrete problem,
we prove optimal error bounds for the semi-discrete solution and Monte-Carlo samplings of its expectation 
in appropriate Bochner spaces.
Our theoretical findings are illustrated by numerical experiments  in two and three space dimensions.
\end{abstract}


\maketitle

\section{Introduction}
Surface partial differential equations, i.e.,  partial differential equations on stationary or evolving surfaces,
have become a flourishing mathematical field 
with numerous applications, e.g., in  image processing \cite{Image}, computer graphics \cite{BerCheOsh01}, 
cell biology \cite{ElliottStinnerVenkataraman12,PlaSanPad04},  and porous media~\cite{MM}.
The numerical analysis of surface partial differential equations can be traced back to the pioneering paper of Dziuk~\cite{Dz88}
on the Laplace-Beltrami equation.  Meanwhile  there are various extensions to moving hypersurfaces
such as, e.g., evolving surface finite element methods~\cite{DE07,DE13} or
trace finite element methods~\cite{ReuskenOlshanskii17},
and an abstract framework for parabolic  equations on evolving Hilbert spaces~\cite{EAS,EASapp}.

Though uncertain parameters are rather the rule than the exception in many applications
and though partial differential equations with random coefficients have been intensively studied over the last years
(cf., e.g., the monographs \cite{Spectmeth} and \cite{LPS}),
the numerical analysis of random surface partial differential equations still appears to be in its infancy.

In this paper, we present random evolving surface finite element methods
for the advection-diffusion equation
\[
  \partial^\bullet u - \nabla_\Gamma (\alpha \nabla_\Gamma u) + u \nabla_\Gamma \cdot \text{v} = f
\]
 on an evolving compact hypersurface $\Gamma(t) \subset \mathbb{R}^n$, $n=2$, $3$,  
 with  a uniformly bounded random coefficient $\alpha$  and deterministic velocity $\text{v}$
 on a compact time intervall $t \in [0,T]$.
 Here  $\partial^\bullet$  denotes the path-wise material derivative and
 $\nabla_\Gamma$ is the tangential gradient.
 While  the analysis and numerical analysis  of random advection-diffusion equations
 is well developed in the flat case~\cite{C,HS,LMM,NT},  
 to our knowledge, existence, uniqueness and regularity results for curved domains
 have been first  derived only recently in~\cite{Dj}.
Following Dziuk \& Elliott~\cite{DE07}, the space discretization is performed by random piecewise linear finite element functions 
on simplicial approximations $\Gamma_h(t)$ of the surface $\Gamma(t)$, $t \in [0,T]$. 
 We present optimal error estimates for the resulting semi-discrete scheme 
 which then provide corresponding error estimates for expectation values and Monte-Carlo approximations.
 Application of efficient solution techniques, such as adaptivity~\cite{Demlow}, 
 multigrid methods~\cite{KornhuberYserentant08},
 and Multilevel Monte-Carlo techniques~\cite{BSC,CST,CGST} is very promising but beyond the scope of this paper. 
 In our numerical experiments we investigate a corresponding fully discrete scheme
 based on an implicit Euler method and observe optimal convergence rates.

The paper is organized as follows. We start  by setting up some notation, the notion of hypersurfaces, function spaces, and material derivatives in order to derive a weak formulation of  our problem according to~\cite{Dj}.
Section~\ref{evolving polyh} is devoted to  the random ESFEM  discretization in the spirit of~\cite{DE07}
leading to the precise formulation and  well-posedness  
of our semi discretization in space presented in Section~\ref{sec:EVOLF}. 
Optimal error estimates for the approximate solution, its expectation and a Monte-Carlo approximation 
are contained in Section~\ref{sec:ERROREST}. 
The paper concludes with numerical experiments in two and three space dimensions
suggesting that our optimal error estimates extend to corresponding fully discrete schemes.

\section{Random advection-diffusion equations on evolving hypersurfaces}
Let  $(\Omega, \mathcal{F}, \mathbb{P})$ be a complete probability space with sample space $\Omega$, a $\sigma$-algebra of events $\mathcal{F}$ and a probability $\mathbb{P} \colon  \mathcal{F}  \rightarrow [0,1]$. In addition, we assume that $L^2(\Omega)$ is a separable space. For this assumption it suffices to assume that $(\Omega, \mathcal{F}, \mathbb{P})$  is separable \cite[Exercise 43.(1)]{Hal}. 
We consider a fixed finite time interval $[0,T]$, where $T \in (0,\infty).$ Furthermore, 
we denote by $\mathcal{D}((0,T); V)$   
the space of infinitely differentiable functions with values in a a Hilbert space
$V$ and compact support in $(0, T)$.

\subsection{Hypersurfaces}
We first recall some basic notions  and results concerning hypersurfaces and Sobolev spaces on hypersurfaces.
We refer to  \cite{DDE} and  \cite{DE13a} for more details.

Let $\Gamma \subset \R^{n+1}$ $(n=1,2)$ be a $\mathcal{C}^3$-compact, connected, orientable, $n$-dimensional hypersurface without boundary.
For a function $f \colon \Gamma \to \mathbb{R}$ allowing for 
a differentiable extension $\tilde{f}$
to  an open neighbourhood of $\Gamma$ in $\R^{n+1}$ we define 
the \textit{tangential gradient} by
\begin{equation} \label{eq:GRADDEF}
\nabla_\Gamma f(x) := \nabla \tilde{f}(x) - \nabla \tilde{f}(x) \cdot \nu(x) \nu(x),  \quad x \in \Gamma,
\end{equation}
where $\nu(x)$ denotes the unit normal to $\Gamma$.

Note that $\nabla_\Gamma f(x)$ is the orthogonal projection of $\nabla \tilde{f}$
onto the tangent space to $\Gamma$ at $x$ (thus a tangential vector). It depends only on the values of $\tilde{f}$ on $\Gamma$ \cite[Lemma 2.4]{DE13a}, which makes the definition \eqref{eq:GRADDEF} independent of the extension $\tilde{f}$. The tangential gradient is a vector-valued quantity and for its components we use the notation $\nabla_\Gamma f(x) = (\underline{D}_1f(x), \dots, \underline{D}_{n+1}f(x)).$
The \textit{Laplace-Beltrami} operator is defined by
\[
\Delta_\Gamma f(x) = \nabla_\Gamma \cdot \nabla_\Gamma f(x) =
 \sum_{i=1}^{n+1}\underline{D}_i \underline{D}_i f(x), \quad x \in \Gamma.
 \]

 In order to prepare  weak formulations of PDEs on $\Gamma$, 
 we now introduce Sobolev spaces on surfaces. To this end, let $L^2(\Gamma)$ denote the 
 Hilbert space
 of all measurable functions $f \colon \Gamma \rightarrow \R$ such that
$\| f \|_{L^2(\Gamma)} := \left(\int_\Gamma |f(x)|^2\right)^{1/2}$ is finite.
We say that a function $f \in L^2(\Gamma)$ has a weak partial derivative $g_i = \underline{D}_i f \in L^2(\Gamma), \, (i=\{1,\dots,n+1\})$, 
if for every function $\phi \in \mathcal{C}^1(\Gamma)$ and every $i$ there holds
\[
  \int_\Gamma f \underline{D}_i \phi = -\int_\Gamma \phi g_i + \int_\Gamma f \phi H \nu_i
\]
where $H=  - \nabla_\Gamma  \cdot  \nu$ denotes the mean curvature.
The Sobolev space $H^1(\Gamma)$ is then defined by
\[
H^1(\Gamma) = \{ f \in L^2(\Gamma) \mid \underline{D}_i  f \in L^2(\Gamma), \; i=1,\dots,n+1 \}
\]
with the norm
$\|f\|_{H^1(\Gamma)} = (\|f\|_{L^2(\Gamma)}^2 + \|\nabla_\Gamma f\|_{L^2(\Gamma)}^2)^{1/2}$.

For a description of evolving hypersurfaces we consider two approaches,
starting with evolutions according to a given velocity field $\w$. 
Here, we assume that  $\Gamma(t)$ satisfies the same properties 
as $\Gamma(0)=\Gamma$ for every $t \in [0,T]$, and we set $\Gamma_0 := \Gamma(0)$. 
Furthermore, we assume the existence of a flow, i.e.,  of a diffeomorphism
\[
\Phi^0_t(\cdot):={\Phi}(\cdot,t) \colon   \Gamma_0 \rightarrow \Gamma(t), \quad 
\Phi \in \mathcal{C}^1([0,T], \mathcal{C}^1(\Gamma_0)^{n+1}) \cap \mathcal{C}^0([0,T], \mathcal{C}^3(\Gamma_0)^{n+1}),
\]
that satisfies
\begin{equation} \label{VELOCITY}
\frac{d}{dt}\Phi^0_t(\cdot) = \w(t,\Phi^0_t(\cdot)), \qquad
\Phi^0_0(\cdot) = \text{Id}(\cdot),
\end{equation}
with a  $\mathcal{C}^2$-velocity field $\w \colon [0,T] \times \R^{n+1} \rightarrow \R^{n+1}$
with uniformly bounded divergence
\begin{equation}
| \nabla_{\Gamma(t)} \cdot \w(t)| \leq C \quad \forall t \in [0,T]. \label{bound of velocity}
\end{equation}

It is  sometimes convenient to alternatively represent $\Gt$ as the zero level set
of a suitable function defined on a subset of the ambient space $\mathbb{R}^{n+1}$.
More precisely, under the given regularity assumptions for $\Gt$,  it follows by the Jordan-Brouwer theorem 
that $\Gt$ is the boundary of an open bounded domain. 
Thus, $\Gt$ can be represented as the zero level set 
\begin{equation*}
\Gt = \{ x \in \mathcal{N}(t) \,| \,d(x,t) = 0\}, \quad t \in [0,T],
\end{equation*}
of a signed distance function $d=d(x,t)$ 
defined on  an open neighborhood  $\mathcal{N}(t)$ of $\Gt$ such that $|\nabla d| \neq 0$ for $t \in [0,T]$.
Note that  $d$, $d_t$, $d_{x_i}$, $d_{x_ix_j} \in \mathcal{C}^1(\mathcal{N}_T)$ with $i$, $j = 1, \dots, n+1$
holds for
\[
\mathcal{N}_T := \bigcup_{t\in[0,T]} \mathcal{N}(t) \times \{t\}.
\] 
We also choose $\mathcal{N}(t)$ such that for every $x \in \mathcal{N}(t)$ and $t \in [0,T]$
there exists a unique $p(x,t) \in \Gt$ such that
\begin{equation}\label{projection}
x = p(x,t) + d(x,t) \nu (p(x,t),t),
\end{equation}
and fix the orientation of $\Gt$ 
by choosing the normal vector field $\nu(x,t) := \nabla d(x,t)$.
Note that the constant extension of a function 
$\eta (\cdot,t)  \colon \Gt \rightarrow \mathbb{R}$ 
to $\mathcal{N}(t)$ in normal direction is given by 
$\eta^{-l}(x,t) = \eta(p(x,t),t)$, $p \in \mathcal{N}(t)$.
Later on, we will use \eqref{projection}  
to define the lift of functions on approximate hypersurfaces.

\subsection{Function spaces}
In this section, we define Bochner-type function spaces of random functions that are defined on evolving spaces.
The definition of these spaces is taken from \cite{Dj}  and uses the idea from Alphonse et al.~\cite{EAS} 
to map each domain at time $t$  to the fixed initial domain $\Gamma_0$ by a pull-back operator using the flow $\Phi_t^0$.
Note that this approach is similar to Arbitrary Lagrangian Eulerian (ALE) framework.

For each $t \in [ 0, T ]$, let us define
\begin{align}
V(t) &:= {\Vt} \cong L^2(\Omega) \otimes H^1(\Gamma(t))  \label{ten prod}\\
H(t) &:= {\Ht} \cong L^2(\Omega) \otimes L^2(\Gamma(t)) \label{ten prod1}
\end{align}
where the isomorphisms hold because all considered spaces are separable Hilbert spaces (see \cite{RS}). The dual space of $V(t)$ is the space $V^*(t) = L^2(\Omega, H^{-1} (\Gamma(t)))$, where $H^{-1} (\Gamma(t))$ is the dual space of $H^1 (\Gamma(t))$.
Using the tensor product structure of these spaces \cite[Lemma 4.34]{Hack}, it follows that $V(t) \subset H(t) \subset V^*(t)$ is a Gelfand triple for every $t \in [0,T]$.
For convenience we will often (but not always) write $u(\omega,x)$ instead of $u(\omega)(x)$, which is justified by the tensor structure of the spaces.

For an evolving family of Hilbert spaces $X= (X(t))_{t\in[0,T]}$, such as, e.g., $V=(V(t))_{t\in[0,T]}$ or $H=(H(t))_{t\in[0,T]}$ 
we connect the space $X(t)$ for fixed $t \in [0,T]$ with the initial space $X(0)$ by using a family of
so-called pushforward maps $\phi_t \colon X(0) \rightarrow X(t)$, satisfying certain compatibility conditions stated in 
\cite[Definition 2.4]{EAS}.
More precisely, we use its inverse map $\phi_{-t} : X(t) \rightarrow X(0)$, called pullback map,
to define general Bochner-type spaces of functions defined on evolving spaces as  follows 
(see \cite{EAS,Dj})
\begin{align*}
L^2_X &:= \left \{ u: [0,T] \ni t \mapsto (\bar{u}(t),t) \in  \!\!\! \underset{t\in[0,T]}{\bigcup}\!
X(t) \times \{t\}  
\mid \phi_{- (\cdot) } \bar{u}(\cdot ) \in L^2(0,T;X(0)) \right \},  \\
L^2_{X^*} &:= \left \{ f: [0,T]  \ni t  \mapsto  (\bar{f}(t),t) \in  \!\!\! \underset{t\in[0,T]}{\bigcup}\!X^*(t) \times \{t\}
\mid 
\phi_{- (\cdot)}\bar{f}(\cdot) \in L^2(0,T;X(0)^*) \right \} .
\end{align*}
In the following we will identify $u(t) = (\overline{u}(t); t)$ with $\overline{u}(t)$.

From \cite[Lemma 2.15]{EAS} it follows that $L^2_{X^*}$ and $(L^2_X)^*$ are isometrically isomorphic.
The spaces $L^2_X$ and $L^2_{X^*}$ are separable Hilbert spaces \cite[Corollary 2.11]{EAS} with the inner product defined as
\[
(u,v)_{L^2_X} = \int_{0}^{T}(u(t),v(t))_{X(t)} \, \mathrm{d} t \quad
(f,g)_{L^2_{X^*}} = \int_{0}^{T}(f(t),g(t))_{X^*(t)} \, \mathrm{d} t.
\]

For the evolving family $H$ defined in  \eqref{ten prod1}
we define the pullback operator  $\phi_{-t}: H(t)  \rightarrow H(0)$ for fixed $t\in [0,T]$ and each $u \in H(t)$ by 
\begin{equation*}
(\phi_{-t}u)(\omega, x) := u(\omega, \Phi_t^0(x)), \qquad  x \in \Gamma_0=\Gamma(0), \; \omega \in \Omega,
\end{equation*}
utilizing the parametrisation $\Phi_t^0$ of $\Gamma(t)$ over $\Gamma_0$. 
Exploiting $V(t)\subset H(t)$, the pullback operator  $\phi_{-t}: V(t) \rightarrow V(0)$ is defined by restriction.
It follows from \cite[Lemma 3.5]{Dj}  that the resulting spaces $L^2_V$, $L^2_{V^*}$ and $L^2_H$ are well-defined and
\begin{equation*}
L^2_V \subset L^2_H \subset L^2_{V^*}
\end{equation*}
is a Gelfand triple.

\subsection{Material derivative}
Following~\cite{Dj}, we introduce a material derivative of sufficiently smooth random functions that takes spatial movement into account.

First let us define the spaces of  pushed-forward continuously differentiable functions
\[
 \mathcal{C}_{X}^j := \{ u \in L^2_{X}  \mid
 \phi_{-(\cdot)}u(\cdot) \in \mathcal{C}^j\left([0,T],X(0)\right)
 \} \quad \text{for } j \in \{0, 1, 2\}.
\]
For $u \in \mathcal{C}_V^1$ the material derivative $\md u \in \mathcal{C}^0_V$ is defined by
\begin{equation}\label{material derivative}
\md{u} := \phi_t \left(\frac{d}{dt}\phi_{-t}u \right) = u_t + \nabla u  \cdot \w.
\end{equation}
More precisely, the material derivative of $u$ is defined via a smooth extension   $\tilde{u}$ of $u$  to $\mathcal{N}_T$ 
with well-defined derivatives $\nabla  \tilde{u}$ and $\tilde{u}_t$ and subsequent restriction to 
\[
\mathcal{G}_T := \bigcup_t \Gt \times \{t\} \subset \mathcal{N}_T.
\]
Since,  due to the smoothness of $\Gamma(t)$ and $\Phi_0^t$,
this definition is independent of the choice of particular extension  $\tilde{u}$, 
we simply write $u$ in \eqref{material derivative}.

\begin{xrem}\label{rem:WS}
Replacing classical derivatives  in time by weak derivatives
leads to a weak material derivative $\md{u}\in L^2_{V*}$. It coincides
with the strong material derivative for sufficiently smooth functions.
As we will concentrate on the smooth case later on, 
we omit a precise definition here and refer 
to \cite[Definition 3.9]{Dj} for details.
\end{xrem}

\subsection{Weak formulation and well-posedness} \label{subsec:WFWP}

We consider an initial value problem for an advection-diffusion equation on the evolving surface $\Gamma(t)$, $t \in [0,T]$,
which in strong form reads
\begin{equation}\label{strong form}
\begin{aligned}
\partial^\bullet u - \Ng \cdot (\alpha \Ng u) + u \Ng \cdot \w &= f &&   \\
u(0) &= u_0
\end{aligned} .
\end{equation}
Here the diffusion coefficient $\alpha$ and the initial function $u_0$ are random functions,
and we set $f\equiv 0$ for ease of presentation.

We will consider weak solutions of \eqref{strong form} from the space
\begin{equation} \label{solsp}
W(V,H):= \{ u \in L^2_V \, | \,\md{u} \in L^2_H \}
\end{equation}
where $\md{u}$ stands for the weak material derivative.
$W(V,H)$ is a separable Hilbert space with the inner product defined by
\begin{equation*}
(u,v)_{W(V,H)} = \int_0^T \int_\Omega (u, v)_{H^1(\Gamma(t))} + \int_0^T \int_\Omega (\md{u}, \md{v})_{L^{2}(\Gamma(t))}.
\end{equation*}

Now a \emph{a weak solution}  of (\ref{strong form}) is a solution of the following problem.
\begin{problem}[Weak form of the random advection-diffusion equation on $\{\Gamma(t)\}$] \label{prob:WF1}
Find $u \in W(V,H)$ that point-wise satisfies the initial condition $u(0)=u_0 \in V(0)$ and  
\begin{equation} \label{weak form}
\int_\Omega \! \int_{\Gamma(t)} \!\! \partial^\bullet u(t) \varphi + \int_\Omega \! \int_{\Gamma(t)}\!\! \alpha(t) \Ng u(t) \cdot \Ng \varphi + \int_\Omega \! \int_{\Gamma(t)} \!\!\! u(t) \varphi \Ng \cdot \w (t) =0,
 \end{equation}
for every $\varphi \in L^2(\Omega, H^1(\Gamma(t)))$ and a.e.  $t \in [0,T]$.
\end{problem}

Existence and uniqueness can be stated on the following assumption.
\begin{ass} \label{ass1}  The diffusion coefficient $\alpha$ satisfies the following conditions
\begin{itemize}
\item[a)]  $\alpha \colon \Omega \times\mathcal{G}_T \rightarrow \R$  is a $\mathcal{F} \otimes \mathcal{B}(\mathcal{G}_T )$-measurable.

\item[b)] $\alpha(\omega, \cdot, \cdot) \in \mathcal{C}^1(\mathcal{G}_T)$ holds for $\mathbb{P}$-a.e $\omega \in \Omega$,  
which implies boundedness of $|\md{\alpha}(\omega)|$ on $\mathcal{G}_T$,
and we assume that this bound is uniform in $\omega\in \Omega$. 

\item[c)]  $\alpha$ is uniformly bounded from above and below in the sense 
that there exist  positive constants $\alpha_{\min}$ and $\alpha_{\max}$ such that 
\begin{equation}\label{uniform}
0 < \alpha_{\min} \leq \alpha (\omega,x,t) \leq \alpha_{\max} < \infty \quad \forall (x,t) \in \mathcal{G}_T
\end{equation}
holds for $\mathbb{P}$-a.e. $\omega \in \Omega$
\end{itemize}
and the initial function satisfies $u_0 \in L^2(\Omega, H^1(\Gamma_0))$.
\end{ass}

The following proposition is a consequence of  \cite[Theorem 4.9]{Dj}.
\begin{proposition} \label{prop:EXUNIQ}
Let Assumption \ref{ass1} hold. 
Then, under the given assumptions on $\{\Gamma(t)\}$, 
there is a unique solution $u \in W(V,H)$ 
of Problem~\ref{prob:WF1}  and we have the \textit{a priori} bound 
\begin{equation*}
\|u \|_{W(V,H)} \leq C \| u_0 \|_{V(0)}
\end{equation*} 
with   some $C\in \mathbb{R}$.
\end{proposition}

The following assumption of the diffusion coefficient 
will ensure regularity of the solution.

\begin{ass} \label{ass2}
Assume that there exists a constant $C$ independent of $\omega\in \Omega$ such that
\[
|\Ng \alpha(\omega, x,t)| \leq C \quad \forall (x,t) \in \mathcal{G}_T
\]
holds for $\mathbb{P}$-almost all $\omega \in \Omega$.
\end{ass}

Note that (\ref{uniform}) and Assumption \ref{ass2} 
imply that $\|\alpha(\omega,t)\|_{C^1(\Gamma(t))}$
is uniformly bounded in $\omega\in \Omega$. This will be used later to prove an $H^2(\Gamma(t))$ bound.
In the subsequent error analysis, we will assume further that 
$u$ has a path-wise strong material derivative,
 i.e. that  $u(\omega) \in C^1_{V}$ holds for all $\omega \in \Omega$.

In order to derive a more convenient formulation of  Problem~\ref{prob:WF1} 
with identical solution and test space,
we introduce the time dependent bilinear forms
\begin{equation} \label{eq:CONTBIL}
\begin{array}{cc}
 \displaystyle  m(u,\varphi) := \! \int_\Omega \! \int_{\Gamma(t)} \!\!\!u  \varphi, \quad
 g(\w;u,\varphi) :=\! \int_\Omega \!\int_{\Gamma(t)} \!\!\!u\varphi \Ng \cdot \w , \\[4mm]
  \displaystyle a(u,\varphi) := \! \int_\Omega \! \int_{\Gamma(t)} \!\!\! \alpha \Ng u \cdot \Ng \varphi,  \quad
  b(\w;u,\varphi) := \int_\Omega \int_{\Gamma(t)} B(\omega, \w) \Ng u \cdot \Ng \varphi 
\end{array}
\end{equation}
for $u, \varphi \in L^2(\Omega, H^1(\Gamma(t)))$ and each $t\in [0,T]$.
The tensor $B$ in the definition of $  b(\w;u,\varphi)$ takes the form
\[
B(\omega, \w) =(\md{\alpha} + \alpha \Ng \cdot \w)\text{Id}  - 2\alpha D_\Gamma(\w)
\]
with ${\rm Id}$ denoting the identity in $(n+1)\times (n+1)$ and $(D_\Gamma \w)_{ij} = \underline{D}_j \w^i$. 
Note that (\ref{bound of velocity}) and  the uniform  boundedness 
of  $\md{\alpha}$ on $\mathcal{G}_T$ imply that
$|B(\omega,\w)| \leq C$ holds $\mathbb{P}$-a.e. $\omega\in \Omega$
with some $C \in \mathbb{R}$.

The transport formula for the differentiation of the time dependent surface integral then reads (see e.g. \cite{Dj})
\begin{align}\label{eq:TRANSP}
\frac{d}{dt}m(u,\varphi) &= m(\md{u},\varphi) + m(u,\md{\varphi}) + g(\w; u, \varphi) ,
\end{align}
where the  equality holds a.e.~in $[0,T]$.
As a consequence of  \eqref{eq:TRANSP},   Problem~\ref{prob:WF1} 
is equivalent to the following formulation with identical solution and test space.
\begin{problem}[Weak form of the random advection-diffusion equation on $\{\Gamma(t)\}$] \label{prob:WF2}
Find $u \in W(V,H)$ that point-wise satisfies the initial condition $u(0)=u_0 \in V(0)$ and 
\begin{equation}\label{3.8}
\frac{d}{dt}m(u,\varphi) + a(u,\varphi) = m(u,\md{\varphi}) \qquad \forall \varphi \in W(V,H).
\end{equation}
\end{problem}

This formulation will be used in the sequel.

\section{Evolving simplicial surfaces}\label{evolving polyh}
As a first step towards a discretization of the weak formulation~\eqref{3.8} 
we now consider simplicial approximations
of the evolving surface $\Gamma(t)$, $t\in [0,T]$.
Let $\Gamma_{h,0}$ be an approximation of $\Gamma_0$
consisting of nondegenerate simplices $\{E_{j,0}\}_{j=1}^N =: {\mathcal T}_{h,0}$ 
with vertices $\{X_{j,0}\}_{j = 1}^J \subset \Gamma_0$
such that the intersection of two different simplices 
is a common lower dimensional simplex or empty. 
For $t \in [0,T]$, we let the vertices $X_j(0) = X_{j,0}$
evolve with the smooth surface velocity  $X'_j(t) = \w(X_j(t),t)$, $j=1,\dots,J$,  
and consider the approximation $\Gamma_h(t)$ of $\Gamma(t)$
consisting of the corresponding simplices $\{E_j(t)\}_{j=1}^M =: {\mathcal T}_{h}(t)$.
We assume that
shape regularity of  ${\mathcal T}_{h}(t)$ holds uniformly in $t\in [0,T]$ and that 
${\mathcal T}_{h}(t)$ is quasi-uniform, uniformly in time, in the sense that
\[
h := \sup_{t \in (0,T)} \max_{E(t) \in \mathcal{T}_h(t)} \text{diam} \, E(t) \geq  \inf_{t \in (0,T)}  \min_{E(t) \in \mathcal{T}_h(t)} \text{diam} \,  E(t) \geq c h
\]
holds with some $c \in \R$.
We also assume that  $\Gamma_h(t)\subset \mathcal{N}(t)$ for $t\in [0,T]$ 
and, in addition to \eqref{projection}, that for every $p\in \Gt$ there is a unique
$x(p,t) \in \Ght$ such that 
\begin{equation} \label{eq:PROPP}
 p= x(p,t)+ d(x(p,t),t)\nu(p,t).
\end{equation}
Note  that $\Ght$ can be considered as interpolation of $\Gt$ in $\{X_{j}(t)\}_{j = 1}^J$
and a discrete analogue of the space time domain $\mathcal{G}_T$ is given by
  \begin{equation*}
    \mathcal{G}_T^h := \bigcup_t \Ght \times \{ t \}.
  \end{equation*}

We define the  tangential gradient of a 
 sufficiently smooth function $\eta_h\colon \Ght  \rightarrow \mathbb{R}$ in an element-wise sense, i.e., we set
\begin{equation*}
\nabla_{\Gamma_h} \eta_h |_E = \nabla \eta_h - \nabla \eta_h \cdot \nu_h \nu_h, \qquad E \in \mathcal{T}_h(t).
\end{equation*}
Here $\nu_h$ stands for the element-wise  outward unit normal to  $E\subset \Gamma_h(t)$.
We use the notation $\nabla_{\Gamma_h} \eta_h = (\underline{D}_{h,1}\eta_h, \dots, \underline{D}_{h,n+1}\eta_h)$.

We define the discrete velocity $V_h$ of $\Ght$ by interpolation 
of the given velocity~$\w$, i.e. we set
\[
V_h(X(t),t) := \tilde{I}_h \w(X(t),t), \qquad X(t) \in \Gamma_h(t),
\]
with $\tilde{I}_h$ denoting piecewise linear interpolation in $\{X_{j}(t)\}_{j = 1}^J$. 

We consider the Gelfand triple on $\Gamma_h(t)$
\begin{equation} \label{eq:DGEL}
L^2(\Omega, H^1(\Gamma_h(t))) \subset L^2(\Omega, L^2(\Gamma_h(t))) \subset L^2(\Omega, H^{-1}(\Gamma_h(t)))
\end{equation}
and denote
\[
\mathcal{V}_h(t) := L^2(\Omega, H^1(\Gamma_h(t))) \quad\text{and}\quad
\mathcal{H}_h(t) := L^2(\Omega, L^2(\Gamma_h(t))).
\]
As in the continuous case,  this leads to the following Gelfand triple of evolving Bochner-Sobolev spaces
\begin{equation} \label{discGT}
L^2_{\mathcal{V}_h(t)} \subset L^2_{\mathcal{H}_h(t)} \subset L^2_{\mathcal{V}^*_h(t)}. 
\end{equation}

The discrete velocity $V_h$ induces a discrete strong material derivative
in terms of an element-wise version of (\ref{material derivative}), i.e.,
for sufficiently smooth  functions  $\phi_h \in L^2_{\mathcal{V}_h}$ and any $E(t) \in \Ght$
we set
\begin{equation} \label{eq:DSMD}
\partial_h^\bullet \phi_h |_{E(t)} := (\phi_{h,t} + V_h \cdot \nabla \phi_h) |_{E(t)}.
\end{equation}

We define discrete analogues to the bilinear forms introduced in 
\eqref{eq:CONTBIL} on $\mathcal{V}_h(t) \times \mathcal{V}_h(t)$ according to
\begin{multline*}
m_h(u_h, \varphi_h) := \int_\Omega \int_{\Gamma_h(t)} \!\!u_h \varphi_h, \qquad
g_h(V_h;u_h, \varphi_h) := \int_\Omega \int_{\Gamma_h(t)} u_h \varphi_h\nabla_{\Gamma_h} \cdot V_h,\\
a_h(u_h, \varphi_h) := \int_\Omega \int_{\Gamma_h(t)} \!\! \alpha^{-l} \nabla_{\Gamma_h}u_h \cdot \nabla_{\Gamma_h}\varphi_h, \\
b_h(V_h;\phi,U_h) :=\sum\limits_{E(t) \in \mathcal{T}_h(t)} \int_\Omega \int_{E(t)} B_h(\omega, V_h) \Ngh \phi \cdot \Ngh U_h
\end{multline*}
involving  the tensor
\begin{equation*}
B_h(\omega, V_h) = (\mdh {\alpha^{-l}} + \alpha^{-l} \Ngh \cdot V_h){\rm Id}  - 2\alpha^{-l} D_h(V_h)
\end{equation*}
denoting $(D_h(V_h))_{ij} = \underline{D}_{h,j} V_h^i$.
  Here, we denote  
  \begin{equation}\label{invlift}
  \alpha^{-l}(\omega,x,t) := \alpha(\omega, p(x,t), t) \quad \omega \in \Omega, \, \ (x,t) \in \mathcal{G}_T^h
  \end{equation}
 exploiting $\{\Gamma_h(t)\}\subset {\mathcal N}(t)$ and \eqref{projection}. 
 Later $\alpha^{-l}$ will be called the inverse lift of $\alpha$. 

Note that $\alpha^{-l}$ satisfies a discrete version of Assumption~\ref{ass1} and \ref{ass2}.
In particular, $\alpha^{-l}$ is an $\mathcal{F} \otimes \mathcal{B}({\mathcal{G}_T^h})$-measurable
function, $\alpha^{-l}( \omega, \cdot, \cdot )|_{E_T} \in \mathcal{C}^1(E_T)$ for all 
space-time elements $E_T := \bigcup_{t} E(t) \times \{ t \}$, and 
$\alpha_{\min} \le \alpha^{-l}( \omega, x, t ) \le \alpha_{\max}$ 
for all $\omega \in \Omega$, $(x,t) \in \mathcal{G}_T^h$.

The next lemma provides a uniform bound for the divergence of $V_h$ 
and the norm of the tensor $B_h$ 
that follows from the geometric properites of $\Gamma_h(t)$ in analogy to \cite[Lemma~3.3]{ER}.

\begin{lemma}\label{boundBh}
Under the above assumptions on $\{\Ght\}$, it holds
\begin{equation*}
\sup_{t \in [0,T]} \left( \| \Ngh \cdot V_h \|_{L^ \infty(\Ght)}  + \| B_h\|_{L^2(\Omega,L^ \infty(\Ght))} \right) \leq c \sup_{t \in [0,T]} \| \w(t)\|_{\mathcal{C}^2 (\mathcal{N}_T)}
\end{equation*}
with a constant $c$ depending only on  the initial hypersurface $\Gamma_0$ and 
the uniform shape regularity and quasi-uniformity of ${\mathcal T}_h(t)$.
\end{lemma}

Since the probability space does not depend on time, the discrete analogue of the
corresponding transport formulae hold, where the discrete material velocity 
and discrete tangential gradients are understood in an element-wise sense. 
The resulting discrete result is stated for example in \cite[Lemma 4.2]{DE13}. 
The following lemma follows by integration over $\Omega$.

\begin{lemma}[Transport lemma for triangulated surfaces]   \label{lem:PT}
Let $\{\Ght\}$ be a family of triangulated surfaces evolving with discrete velocity $V_h$. Let $\phi_h, \eta_h$ be time dependent functions such that the following quantities exist. Then
\begin{equation*}
\frac{d}{dt} \int_\Omega \int_{\Gamma_h(t)} \phi_h = \int_\Omega \int_{\Gamma_h(t)} \mdh \phi_h + \phi_h \nabla_{\Gamma_h} \cdot V_h.
\end{equation*}
In particular, 
\begin{equation} \label{eq:TPF1}
\frac{d}{dt}m_h(\phi_h,\eta_h) = m(\mdh \phi_h,\eta_h) + m(\phi_h,\mdh \eta_h) + g_h(V_h; \phi_h, \eta_h). 
\end{equation}
\end{lemma}

\section{Evolving surface finite element methods} \label{sec:EVOLF}
Following~\cite{DE07}, we now introduce an evolving 
surface finite element discretization  (ESFEM) of Problem~\ref{prob:WF2}.
\subsection{Finite elements on simplicial surfaces} 
For each $t \in [0,T]$ we define the \textit{evolving  finite element space}
\begin{equation}
S_h(t)  := \{ \eta \in \mathcal{C}(\Gamma_h(t)) \; | \, \eta_{E} \, \text{ is affine }\forall E \in \mathcal{T}_h(t) \} . \label{FEM}\\
\end{equation}
We denote by $\{ \chi_j(t)\}_{ j=1, \dots, J }$ the nodal basis of $S_h(t)$, i.e.  $\chi_j(X_i(t),t) = \delta_{ij}$ (Kronecker-$\delta$).
These basis functions  satisfy the transport property \cite[Lemma 4.1]{DE13}
\begin{equation}
\partial_h^\bullet \chi_j = 0 . \label{tr prop}
\end{equation}

We consider the following Gelfand triple 
\begin{equation} \label{disGel}
S_h(t) \subset L_h(t) \subset S^*_h(t), 
\end{equation}
where all three spaces algebraically coincide  but are equipped with different norms
inherited from the corresponding continuous counterparts, i.e.,
\[
S_h(t) := ( S_h(t), \| \cdot \|_{H^1(\Gamma_h(t))} ) \quad \text{and} \quad L_h(t) := (S_h(t), \| \cdot \|_{L^2(\Gamma_h(t))}).
\]
The dual space $S^*_h(t)$ consists of all continuous linear functionals on $S_h(t)$ and is equipped with the standard dual norm 
\[
\| \psi \|_{S^*_h(t)} := \sup_{\{\eta \in S_h(t) \;| \; \|\eta\|_{H^1(\Gamma_h(t))}=1\}}  | \psi(\eta) |. 
\]
Note that all three norms are equivalent as norms on finite dimensional spaces, which implies that (\ref{disGel}) is the Gelfand triple. 
As a discrete counterpart of~\eqref{eq:DGEL}, we introduce the Gelfand triple
\begin{equation}
L^2(\Omega, S_h(t)) \subset L^2(\Omega, L_h(t)) \subset L^2(\Omega, S^*_h(t)) \label{discreteGt}.
\end{equation}
Setting
\[
V_h(t) :=L^2(\Omega, S_h(t)) \quad
H_h(t) := L^2(\Omega, L_h(t)) \quad
V^*_h(t) := L^2(\Omega, S^*_h(t))
\]
we obtain the finite element analogue 
\begin{equation}
L^2_{V_h(t)} \subset L^2_{H_h(t)}  \subset L^2_{V^*_h(t)}
\end{equation}
of the Gelfand triple \eqref{discGT} of evolving Bochner-Sobolev spaces.
Let us note that since the sample space $\Omega$ is independent of time, it holds
\begin{equation}
L^2(\Omega, L^2_X) \cong L^2(\Omega) \otimes L^2_X \cong L^2_{L^2(\Omega, X)} \label{pathisom}
\end{equation}
for any evolving family of separable Hilbert spaces $X$ (see, e.g., Section {\ref{evolving polyh}). 
We will exploit this isomorphism for $X=S_h$ in the following definition of the solution space for the semi-discrete problem, where we will rather consider the problem in a path-wise sense.

We define the solution space for the semi-discrete problem as the space of functions 
that are smooth for each path in the sense that
$\phi_h(\omega) \in \mathcal{C}^1_{S_h}$ holds for all  $\omega \in \Omega$. 
Hence, $\partial^\bullet_h \phi_h$ is defined path-wise 
for path-wise smooth functions.  
In addition, we require $\partial^\bullet_h \phi_h(t) \in H_h(t)$
to define the semi-discrete solution space
\[
W_h(V_h,H_h):= L^2(\Omega, \mathcal{C}^1_{S_h}).
\]
The scalar product of this space is defined by
\[
(U_h, \phi_h)_{W_h(V_h,H_h)} := \int_0^T  \int_{\Omega} (U_h,\phi_h )_{H^1(\Gamma_h(t))} + \int_0^T  \int_{\Omega} (\mdh U_h, \mdh \phi_h)_{L^2(\Gamma_h(t))}
\]
with the associated norm $\|\cdot \|_{W_h(V_h,H_h)}$. 

The semi-discrete approximation of Problem~\ref{prob:WF2},
on $\{\Gamma_h(t)\}$ now reads as follows.
\begin{problem}[ESFEM discretization in space] \label{prob:ESFEM}
Find $U_h \in W_h(V_h,H_h)$ that point-wise satisfies the initial condition $U_h(0)=U_{h,0} \,\in V_h(0)$ and 
\begin{equation}\label{semidiscrete problem}
\frac{d}{dt}m_h(U_h,\varphi) + a_h(U_h,\varphi) = m_h(U_h,\mdh{\varphi})  \qquad \forall \varphi \in W_h(V_h,H_h).
\end{equation}
\end{problem}

In contrast to $W(V,H)$, the semidiscrete space
$W_h(V_h,H_h)$ is not complete so that the proof of the following 
existence and stability result  requires a different kind of argument.

\begin{theorem}\label{theo:SEDIP}
The semi-discrete problem (\ref{semidiscrete problem 2}) has a unique solution $U_h \in W_h(V_h, H_h)$ 
which satisfies the stability property
\begin{equation} \label{eq:STABP}
\| U_h \|_{W(V_h, H_h)} \leq C \| U_{h,0} \|_{V_h(0)}
\end{equation}
with a mesh-independent constant $C$ 
depending only on  $T$, $\alpha_{\min}$, and the bound for $\|\Ngh \cdot V_h\|_{\infty}$ from Lemma~\ref{boundBh}.
\end{theorem}

\begin{proof}
In analogy to  Subsection ~\ref{subsec:WFWP}, Problem ~\ref{prob:ESFEM} is equivalent to
find $U_h \in W_h(V_h,H_h)$ that point-wise satisfies the 
initial condition $U_h(0)=U_{h,0} \,\in V_h(0)$ and
\begin{equation}\label{semidiscrete problem 2}
m_h(\mdh U_h, \varphi) + a(U_h, \varphi) + g(V_h; U_h, \varphi) =0 
\end{equation}
for every $\varphi \in L^2(\Omega, S_h(t))$ and a.e. $t \in [0,T]$.

Let $\omega \in \Omega$ be arbitrary but fixed. We start with 
considering the deterministic path-wise problem
to find $U_h(\omega)\in {\mathcal C}_{S_h}^1$
such that $U_h(\omega;0)=U_{h,0}(\omega)$ and 
\begin{equation}\label{sdpath}
\int_{\Gamma_h(t)} \mdh U_h(\omega) \varphi  + \int_{\Gamma_h(t)} \alpha^{-l}(\omega) \Ngh U_h(\omega) \cdot \Ngh \varphi  + \int_{\Gamma_h(t)} U_h(\omega) \varphi \Ngh \cdot V_h = 0
\end{equation}
holds for all $\varphi \in S_h(t)$ and a.e.\ $t \in [0,T]$.
Following Dziuk \& Elliott \cite[Section 4.6]{DE13},
we insert the nodal basis representation
\begin{equation}\label{sumrep}
U_h(\omega,t,x) = \sum_{j=1}^J U_j(\omega,t) \chi_j(x,t)
\end{equation}
into  (\ref{sdpath}) 
and take $\varphi = \chi_i(t)\in S_h(t), \, i=1,\dots,J$, as test functions.
Now the transport property (\ref{tr prop}) implies
\begin{eqnarray}\label{prematrix}
\sum_{j=1}^J \frac{\partial}{\partial t} U_j(\omega) \int_{\Gamma_h(t)} \!\!\! \chi_j \chi_i &+& \sum_{j=1}^J U_j(\omega) \int_{\Gamma_h(t)} \!\!\! \alpha^{-l}(\omega) \Ngh \chi_j \cdot \Ngh \chi_i \\
&+&
\sum_{j=1}^J U_j(\omega) \int_{\Gamma_h(t)}  \!\!\! \chi_j \chi_i  \Ngh \cdot V_h = 0. \nonumber
\end{eqnarray}
We introduce the evolving mass matrix $M(t)$ with coefficients
\[
M(t)_{ij} := \int_{\Gamma_h(t)} \chi_i(t) \chi_j(t),
\]
and the evolving stiffness matrix $S(\omega, t)$ with coefficients
\[
S(\omega, t)_{ij} := \int_{\Gamma_h(t)}\alpha^{-l}(\omega,t) \Ngh \chi_j(t)\Ngh \chi_i(t).
\] 
From \cite[Proposition 5.2]{DE13} it follows
\[
\frac{dM}{dt} = M'
\]
where 
\[
M'(t)_{ij} :=  \int_{\Gamma_h(t)}  \!\!\! \chi_j(t) \chi_i(t)  \Ngh \cdot V_h(t).
\]
Therefore, we can write (\ref{prematrix}) as the following linear initial value problem
\begin{equation} \label{eq:LINSYS}
   \frac{\partial}{\partial t} (M(t)U(\omega,t) )
  +  S(\omega,t) U(\omega,t) = 0,
  \qquad
  U(\omega,0) = U_0(\omega),
\end{equation}
for the unknown  vector $U(\omega,t) = (U_j(\omega,t))_{i=1}^J$ of coefficient functions.
As in \cite{DE13}, 
there exists an unique path-wise semi-discrete solution $U_h(\omega) \in \mathcal{C}^1_{S_h}$,
since the matrix $M(t)$ is uniformly positive definite on $[0,T]$ and the stiffness matrix $S(\omega, t)$ is positive semi-definite for every $\omega \in \Omega$. Note that the time regularity of $U_h(\omega)$ follows from $M$, $S(\omega) \in C^1(0,T)$
which in turn is a consequence of our assumptions on the time regularity 
of the  evolution of $\Gamma_h(t)$.

The next step is to prove the measurability of the map $\Omega \ni \omega \mapsto U_h(\omega) \in \mathcal{C}^1_{S_h}$. 
On $\mathcal{C}^1_{S_h}$  we consider the Borel $\sigma-$algebra induced by the norm
\begin{equation}\label{c1norm}
\|U_h\|^2_{\mathcal{C}^1_{S_h}} := \int_0^T \|U_h(t) \|^2_{H^1(\Gamma_h(t)) } + \|\mdh U_h(t) \|^2_{L^2(\Gamma_h(t)) }.
\end{equation}
We write (\ref{prematrix}) in the following form
\[
\frac{\partial}{\partial t} U(\omega,t) + A(\omega, t) U(\omega, t) = 0, \qquad
  U(\omega,0) = U_0(\omega),
\]
where 
\[
A(\omega, t) := M^{-1}(t)\left(  M'(t) + S(\omega,t) \right) .
\]

As $U_{h,0}\in V_h(0)$, the function $\omega \mapsto U_{0}(\omega)$ is measurable
and since $\alpha^{-l}$ is a $\mathcal{F} \otimes \mathcal{B}({\mathcal{G}_T^h})$-measurable function, it follows from Fubini's Theorem \cite[Sec.~36, Thm.~C]{Hal} that
\[
 \Omega \ni \omega \mapsto \left(U_{0}(\omega), A(\omega) \right) \in \mathbb{R}^J \times \left(C^1 \left([0,T], \mathbb{R}^J \right), \|\cdot\|_{\infty}\right)
\]
 is measurable function. Utilizing Gronwall's lemma it can be shown that the mapping  
\[
 \mathbb{R}^J \times \left(C^1 \left([0,T], \mathbb{R}^J \right), \|\cdot\|_{\infty}\right) \ni
(U_{0}, A) \mapsto U \in  \left(C^1 \left([0,T], \mathbb{R}^J \right), \|\cdot\|_{\infty}\right)
\]
is continuous. Furthermore, the mapping
\[
  \left(C^1 \left([0,T], \mathbb{R}^J \right), \|\cdot\|_{\infty}\right) \ni U \mapsto U \in  \left(C^1 \left([0,T], \mathbb{R}^J \right), \|\cdot\|_{2}\right)
\]
with
\[
\|U\|_2^2 := \int_0^T \|U(t)\|^2_{\mathbb{R}^J} +  \|\frac{d}{dt}U(t)\|^2_{\mathbb{R}^J}
\]
is continuous. Exploiting that the triangulation ${\mathcal T}_h(t)$ of $\Gamma_h(t)$
is quasi-uniform, uniformly in time, the continuity of the linear mapping 
\[
 \left(C^1\left([0,T], \mathbb{R}^J\right), \|\cdot\|_{2}\right) \ni U \mapsto U_h \in  \mathcal{C}^1_{S_h}
\]
follows from the  triangle inequality and the Cauchy-Schwarz inequality.
We finally conclude that the function
\[
\Omega \ni \omega \mapsto U_h(\omega) \in \mathcal{C}^1_{S_h}
\]
is measurable as a composition of measurable and continuous mappings. 

The next step is to prove the stability property \eqref{eq:STABP}.
For each fixed $\omega \in \Omega$, path-wise stability results from \cite[Lemma 4.3]{DE13}  imply
\begin{equation}\label{pathest}
\| U_h(\omega) \|^2_{\mathcal{C}^1_{S_h}} \leq C \|U_{h,0}(\omega) \|^2_{H^1(\Gamma_h(0))}
\end{equation}
where $C=C(\alpha_{\min}, \alpha_{\max}, V_h, T, \mathcal{G}_h^T)$ 
is independent of $\omega$ and $U_{h,0}(x) \in L^2(\Omega)$. 
Integrating (\ref{pathest}) over $\Omega$ we get the bound 
\[
\| U_h \|_{W(V_h, H_h)}=\| U_h \|_{L^2(\Omega, \mathcal{C}^1_{S_h})}^2 \leq C \|U_{h,0} \|^2_{V_h(0)}.
\]
In particular, we have $U_h \in W_h(V_h,H_h)$. 

It is left to show that $U_h$ solves (\ref{semidiscrete problem 2}) 
and thus Problem~\ref{prob:ESFEM}.
Exploiting the tensor product structure of the test space 
$L^2(\Omega, S_h(t)) \cong L^2(\Omega) \otimes  S_h(t)$ (see \eqref{pathisom}),
we find that 
\[
\{ \varphi_h(x,t) \eta(\omega)\, | \, \varphi_h(t) \in S_h(t), \eta \in L^2(\Omega)\}
\subset L^2(\Omega) \otimes S_h(t)
\]
is a dense subset of $L^2(\Omega, S_h(t))$. 
Taking any test function $\varphi_h(x,t) \eta(\omega)$ from this dense subset,
we first  insert $\varphi_h(x,t)\in S_h(t)$ into the pathwise problem \eqref{sdpath}, 
then multiply with  $\eta(\omega)$, and finally integrate over  $\Omega$ 
to establish \eqref{semidiscrete problem 2}.
This completes the proof.
\end{proof}

\subsection{Lifted finite elements}
We exploit  (\ref{eq:PROPP}) to define the lift $\eta_h^l(\cdot,t)  \colon \Gt \rightarrow \mathbb{R}$
of functions  $\eta_h(\cdot,t)  \colon \Ght \rightarrow \mathbb{R}$  by
\begin{equation*}
\eta_h^l(p,t) := \eta_h (x(p,t)), \quad p \in \Gt.
\end{equation*}
Conversely, \eqref{projection} is utilized to define
the  inverse lift $\eta^{-l}(\cdot,t) \colon \Ght \rightarrow \mathbb{R}$ of functions $\eta(\cdot,t) \colon \Gt \rightarrow \mathbb{R}$ by
\begin{equation*}
\eta^{-l}(x,t) := \eta(p(x,t),t), \quad x \in \Ght.
\end{equation*}
These operators are inverse to each other, i.e.,  $(\eta^{-l})^l=(\eta^{l})^{-l}=\eta$, 
and, taking characteristic functions $\eta_h$, each element $E(t) \in \mathcal{T}_h(t)$ has its unique associated lifted element $e(t) \in \mathcal{T}_h^l(t)$.
Recall that the inverse lift $\alpha^{-1}$ of the diffusion coefficient $\alpha$ 
was already introduced in (\ref{invlift}).

The  next lemma  states equivalence relations 
between corresponding  norms on $\Gt$ and $\Ght$ 
that follow directly from their deterministic  counterparts (see \cite{Dz88}).

\begin{lemma}\label{norm error}
Let $t \in [0,T]$, $\omega \in \Omega$, 
and let $\eta_h(\omega) \colon  \Gamma_h(t) \rightarrow \mathbb{R}$
with the lift $\eta_h^l(\omega) \colon \Gamma \rightarrow \mathbb{R}$.
Then for each plane simplex $E \subset \Gamma_h(t)$ and its curvilinear lift $e \subset \Gamma(t)$, 
there is a constant $c>0$ independent of $E$, h, t, and $\omega$ such that
\begin{align}
  &\frac{1}{c} \, \| \eta_h \|_{L^2(\Omega, L^2(E))} \leq \| \eta_h^l \|_{L^2(\Omega, L^2(e))} \leq c \, \| \eta_h \|_{L^2(\Omega, L^2(E))} \label{5.4}\\
  &\frac{1}{c} \, \| \Ngh \eta_h \|_{L^2(\Omega, L^2(E))} \leq \| \Ng \eta_h^l \|_{L^2(\Omega, L^2(e))} \leq c \, \| \Ngh \eta_h \|_{L^2(\Omega, L^2(E))} \label{5.5} \\
  &\frac{1}{c} \,  \| \nabla^2_{\Gamma_h}\eta_h \|_{L^2(\Omega, L^2(E))} \leq c \| \nabla^2_{\Gamma} \eta_h^l \|_{L^2(\Omega, L^2(e))} + ch  \| \Ng \eta_h^l \|_{L^2(\Omega, L^2(e))}, \label{5.6}
\end{align}
if the corresponding norms are finite.
\end{lemma}

The motion of the vertices of the triangles $E(t) \in \{\mathcal{T}_h(t)\}$ 
induces a discrete velocity $v_h$ of the surface $\{\Gt\}$. 
More precisely, for a given trajectory $X(t)$ of a point on $\{\Gamma_h(t)\}$ 
with velocity $V_h(X(t),t)$
the associated discrete velocity $\w_h$ in $Y(t)=p(X(t),t)$ on $\Gamma(t)$ is defined by
\begin{equation} \label{eq:DVC}
\w_h(Y(t),t) = Y'(t) =  \frac{\partial p}{\partial t}(X(t),t) + V_h(X(t),t) \cdot \nabla p(X(t),t).
\end{equation}
The discrete velocity $\w_h$ gives rise to a discrete material derivative of functions $\varphi \in L^2_{V}$
in an element-wise sense, i.e., we set
\[
\partial_h^\bullet \varphi |_{e(t)} := (\varphi_{t} + \w_h \cdot \nabla \varphi) |_{e(t)}
\]
for all  $e(t) \in  \mathcal{T}_h^l(t)$, where $\varphi_t$ and $\nabla \varphi $ are defined via a smooth extension, analogous to the definition (\ref{material derivative}).

We introduce  a lifted finite element space by
\[
S_h^l(t) :=  \{ \eta^l \in \mathcal{C}(\Gamma(t)) \; | \, \eta \in S_h(t)\}.
\]
Note that there is a unique correspondence 
between each element $\eta \in S_h(t)$ and $\eta^l \in S_h^l(t)$. 
Furthermore, one can show that for every $\phi_h \in S_h(t)$
here holds
\begin{equation} \label{eq:MATLIFT}
\mdh (\phi_h^l) = (\mdh \phi_h)^l.
\end{equation}
Therefore, by (\ref{tr prop}) we get
$$\qquad \partial_h^\bullet \chi_j^l=0.$$
We finally state an analogon to the transport Lemma~\ref{lem:PT} on simplicial surfaces.
\begin{lemma} \label{lem:TPF2}(Transport lemma for smooth triangulated surfaces.)

Let $\Gt$ be an evolving surface decomposed into curved elements $\{\mathcal{T}_h(t)\}$ whose edges move with velocity $\w_h$. Then the following relations hold for functions $\varphi_h, u_h$ such that the following quantities exist
\begin{equation*}
\frac{d}{dt} \int_\Omega \int_{\Gamma(t)} \varphi_h = \int_\Omega \int_{\Gamma(t)} \mdh \varphi_h + \varphi_h \, \nabla_{\Gamma} \cdot \w_h.
\end{equation*}
and
\begin{equation} \label{eq:TPF2}
\frac{d}{dt}m(\varphi,u_h) = 
m(\mdh \varphi_h,u_h) + m(\varphi_h, \mdh u_h) + g(v_h; \varphi_h,u_h).
\end{equation}
\end{lemma}

\begin{xrem} \label{rem:STABLIFT}
Let $U_h$ be the solution of the semi-discrete Problem~\ref{prob:ESFEM} 
with initial condition $U_h(0)=U_{h,0}$
and let $u_h=U_h^l$ with $u_h(0)=u_{h,0}=U_{h,0}^l$ be its lift. Then, as a consequence of Theorem~\ref{theo:SEDIP},
\eqref{eq:MATLIFT}, and Lemma~\ref{norm error}, the following estimate 
\begin{equation}
\|u_h\|_{W(V,H)} \leq C_0 \|u_h(0) \|_{V(0)} \label{apriorilifting}
\end{equation} 
holds with $C_0$  depending on the constants $C$ and $c$ 
appearing in Theorem~\ref{theo:SEDIP} and Lemma~\ref{norm error}, respectively.
\end{xrem}

\section{Error estimates} \label{sec:ERROREST}
\subsection{Interpolation and geometric error estimates}

In this section we formulate the results concerning the approximation of the surface, which are in the deterministic setting proved in \cite{DE07} and \cite{DE13}. Our goal is to prove that they still hold in the random case. The main task is to keep track of constants that appear and show that they are independent of realization. This conclusion mainly follows from the assumption (\ref{uniform}) about the uniform distribution of the diffusion coefficient. Furthermore, we need to show that the extended definitions of the interpolation operator and Ritz projection operator are integrable with respect to $\mathbb{P}$.

We start with an interpolation error estimate for functions 
$\eta \in L^2(\Omega, H^2(\Gamma(t)))$, where the interpolation $I_h \eta$ 
is defined as the lift of piecewise linear nodal interpolation $\widetilde{I}_h\eta \in L^2(\Omega, S_h(t))$.
Note that $\widetilde{I}_h$ is well-defined, 
because the vertices $(X_j(t))_{j=1}^J$ of $\Gamma_h(t)$
lie on the smooth surface $\Gamma(t)$ and $n= 2$, $3$.
\begin{lemma}\label{int estimate}
The interpolation error estimate 
\begin{equation}\label{inter est}
\begin{split}
\| \eta - I_h\eta \|_{H(t)} + h \| \Ng (\eta - I_h\eta) \|_{H(t)} \\
\leq ch^2 \left(\| \nabla_\Gamma^2 \eta \|_{H(t)} + h \| \Ng \eta \|_{H(t)} \right)
\end{split}
\end{equation}
holds for all  $\eta  \in  L^2(\Omega, H^2(\Gamma(t)))$ with a constant $c$ 
depending only on the shape regularity of $\Gamma_h(t)$.
\end{lemma}
\begin{proof}The proof of the lemma follows directly from the deterministic case and Lemma \ref{norm error}.
\end{proof}

We continue with estimating the geometric perturbation errors in the bilinear forms. 

\begin{lemma} \label{bilinear errors}
Let $t\in [0,T]$ be fixed.
For $W_h(\cdot, t)$ and $\phi_h(\cdot, t) \in L^2(\Omega, S_h(t))$ 
with corresponding lifts
$w_h(\cdot, t)$ and $ \varphi_h(\cdot, t) \in L^2(\Omega, S_h^l(t))$
we have the following estimates of the geometric error
\begin{align}
  | m(w_h, \varphi_h) - m_h(W_h, \phi_h) | & \leq c h^2 \|w_h \|_{H(t)} \| \varphi_h \|_{H(t)} \label{5.13}\\
  | a(w_h, \varphi_h) - a_h(W_h, \phi_h) | & \leq c h^2 \|\Ng w_h \|_{H(t)} \| \Ng \varphi_h \|_{H(t)} \label{5.14}\\
  | g(v_h; w_h, \varphi_h) - g_h(V_h; W_h, \phi_h) | & \leq c h^2 \|w_h \|_{V(t)} \| \varphi_h \|_{V(t)} \label{5.15}\\
  |m(\partial_h^\bullet w_h, \varphi_h) - m_h(\partial_h^\bullet W_h, \phi_h)| & \leq
ch^2 \| \partial_h^\bullet w_h \|_{H(t)} \| \varphi \|_{H(t)}. \label{5.22}
\end{align}
\end{lemma}
\begin{proof}
The assertion follows from  uniform bounds of $\alpha(\omega,t)$ and 
$\partial^\bullet_h \alpha(\omega,t)$ with respect to $\omega \in \Omega$
together with  corresponding deterministic results obtained in \cite{DE13} and \cite{LM}.
\end{proof}

Since the velocity $\w$ of $\Gamma(t)$ is deterministic, 
we can use  \cite[Lemma 5.6]{DE13} to control its deviation 
from the discrete velocity  $\w_h$ on $\Gamma(t)$.
Furthermore, \cite[Corollary 5.7]{DE13} provides the following error estimates
for the continuous and discrete material derivative.

\begin{lemma} \label{l5.6}
For the continuous velocity $\w$ of  $\Gamma(t)$ 
and the discrete velocity $\w_h$ defined in \eqref{eq:DVC}
 the estimate
\begin{equation} \label{5.18}
| \w - \w_h | + h \, | \Ng(\w-\w_h)| \leq ch^2 
\end{equation}
holds pointwise on $\Gamma(t)$. Moreover, there holds
\begin{eqnarray}
 &\| \partial^\bullet z -\partial_h^\bullet z\|_{H(t)} 
 \leq  ch^2 \| z \|_{V(t)}, &\quad z \in V(t), \\ \label{5.20}
& \| \Ng(\partial^\bullet z -\partial_h^\bullet z) \|_{H(t)} 
 \leq  ch \| z \|_{L^2(\Omega, H^2(\Gamma))}, &\quad  z \in L^2(\Omega, H^2(\Gamma(t))), \label{5.21}
\end{eqnarray}
provided that the left hand sides are well-defined.
\end{lemma}

\begin{xrem}\label{rm: divh bound}
Since  $\w_h$ is a  ${\mathcal C}^2$-velocity field  by assumption, 
 (\ref{5.18}) implies a uniform upper bound for $\nabla_{\Gamma(t)} \cdot \w_h$
 which in turn yields the estimate
 \begin{equation} \label{eq:GBOUND}
  |g(\w_h;w,\varphi)|\leq c \|w\|_{H(t)}\|\varphi\|_{H(t)},\qquad \forall w,\varphi\in H(t)
 \end{equation}
 with a constant $c$ independent of $h$.
\end{xrem}

\subsection{Ritz projection }
For each fixed $t\in [0,T]$ and $\beta \in L^{\infty}(\Gamma(t))$ with 
$0< \beta_{\min} \leq \beta(x) \leq \beta_{\max} < \infty$ a.e. on $\Gamma(t)$ 
the  Ritz projection 
\[
 H^1(\Gamma(t))\ni v \mapsto \Rh^{\beta} v \in S_h^l(t)
\]
is well-defined by the conditions $\int_{\Gt} \Rh^{\beta} v=0$ and
\begin{equation}\label{pathwise Ritz}
  \int_{\Gamma(t)} \beta \Ng \Rh^{\beta} v \cdot \Ng\varphi_h
  = \int_{\Gamma(t)} \beta \Ng v \cdot \Ng\varphi_h \qquad \forall \varphi_h \in S_h^l(t),
\end{equation}
because $\{\eta \in  S_h^l(t)\;|\; \int_{\Gt} \eta = 0 \}\subset  H^1(\Gt)$ 
is finite dimensional and thus closed. 
Note that
\begin{equation} \label{eq:RITZBOUND}
\|\Ng R^\beta v\|_{L^2(\Gt)} \leq {\textstyle \frac{\beta_{\max}}{\beta_{\min}}}\|\Ng v\|_{L^2(\Gt)}.
\end{equation}

For fixed $t \in [0,T]$, the pathwise Ritz projection $u_p: \Omega \mapsto S_h^l(t)$ 
of $u  \in L^2(\Omega, H^1(\Gt))$ is defined by
\begin{equation} \label{eq:PWRITZ}
\Omega \ni \omega \to u_p(\omega)= R^{\alpha(\omega,t)} u(\omega)\in S_h^l(t).
\end{equation}
In the following lemma, we state regularity and 
$a$-orthogonality. 
\begin{lemma}
Let $t \in [0,T]$ be fixed. Then, 
the pathwise Ritz projection $u_p: \Omega \mapsto S_h^l(t)$ of  $u  \in L^2(\Omega, H^1(\Gt))$
satisfies $u_p \in L^2(\Omega, S_h^l(t))$ and the Galerkin orthogonality
\begin{equation} 
a(u-u_p, \eta_h)=0 \qquad \forall \eta_h \in L^2(\Omega, S_h^l(t)).\label{proj res}
\end{equation}
\end{lemma}

\begin{proof}
By Assumption~\ref{ass1} the mapping
\[
\Omega \ni \omega \mapsto \alpha(\omega,t)\in {\mathcal B}:=\{\beta\in L^{\infty}(\Gamma(t))\;|\; \alpha_{\min}/2 \leq \beta(x) \leq 2\alpha_{\max} \}\subset L^{\infty}(\Gamma(t))
\]
is measurable.
Hence by, e.g., \cite[Lemma A.5]{HLS16}, it is sufficient to prove that the mapping 
\[
{\mathcal B}\ni \beta \mapsto R^{\beta}\in \mathcal{L}(H^1(\Gt), S_h^l(t))
\]
is continuous with respect to the canonical norm in the space $\mathcal{L}(H^1(\Gt), S_h^l(t))$
of linear operators from $H^1(\Gt)$ to $S_h^l(t)$. 
To this end, let $\beta$, $\beta'\in {\mathcal B}$ and $v\in H^1(\Gt)$ be arbitrary
and we skip the dependence on $t$ from now on.
Then, inserting the test  function $\varphi_h=(\Rh^\beta- \Rh^{\beta'})v \in S_h^l(t)$ into the definition  
\eqref{pathwise Ritz}, utilizing the stability \eqref{eq:RITZBOUND},  we obtain 
\begin{eqnarray*}
&\alpha_{\min}/2& \|(\Rh^{\beta'}- \Rh^{\beta})v\|^2_{H^1(\Gamma)} 
\leq (1 + C_P^2) \int_\Gamma \beta |\nabla_\Gamma(\Rh^{\beta'}- \Rh^{\beta})v|^2 \\
&=&(1 + C_P^2) ( \int_\Gamma ( \beta - \beta') \nabla_\Gamma \Rh^{\beta'} v \nabla_\Gamma (\Rh^{\beta'} - \Rh^{\beta})v \\
& +& \int_\Gamma \beta' \nabla_\Gamma \Rh^{\beta'} v  \nabla_\Gamma (\Rh^{\beta'} - \Rh^{\beta})v
-  \int_\Gamma \beta \nabla_\Gamma v \nabla_\Gamma (\Rh^{\beta'} - \Rh^{\beta})v ) \\
 &=&(1 + C_P^2) \left( \int_\Gamma (\beta' - \beta)(\nabla_\Gamma  v   - \nabla_\Gamma \Rh^{\beta'} v )
 \nabla_\Gamma ( \Rh^{\beta'} - \Rh^{\beta})v \right)\\
 &\leq& (1 + C_P^2)  \| \beta' - \beta\|_{L^{\infty}(\Gamma)} \|\nabla_\Gamma(v - \Rh^{\beta'} v)\|_{L^2(\Gamma)}  \|\nabla_\Gamma(\Rh^{\beta'} - \Rh^{\beta})v\|_{L^2(\Gamma)} \\
 &\leq& \left(1+4 \frac{\alpha_{\max}}{\alpha_{\min}}\right)(1 + C_P^2) 
 \| \beta' - \beta\|_{L^{\infty}(\Gamma)}  \|v\|_{H^1(\Gamma)} \|(\Rh^{\beta'} - \Rh^{\beta})v\|_{H^1(\Gamma)},
\end{eqnarray*}
where $C_P$ denotes the Poincar\'{e} constant in $\{\eta \in  H^1(\Gamma)\;|\; \int_{\Gamma} \eta = 0 \}$
(see, e.g., \cite[Theorem 2.12]{DE13a}). 

The norm of $u_p$ in  $L^2(\Omega, H^1(\Gt))$ is bounded, because
Poincar\'{e}'s inequality and \eqref{uniform} lead to
\begin{align*}
\alpha_{\min} \int_\Omega \| u_p(\omega) \|^2_{H^1(\Gt)} 
&\leq (1+ C_P^2) \int_\Omega \alpha(\omega,t) \| \nabla_\Gamma \Rh^{\alpha(\omega,t)}(u(\omega)) \|^2_{L^2(\Gt)} \\ 
\leq (1+ C_P^2) \alpha_{\max} \int_\Omega \| \nabla_\Gamma u(\omega) \|^2_{L^2(\Gt)} 
&\leq (1+ C_P^2) \| \nabla_\Gamma u \|^2_{L^2(\Omega,  H^1(\Gt))}.
\end{align*}
This implies $u_p \in L^2(\Omega, S_h^l(t))$.

It is left to show  \eqref{proj res}.
For that purpose we select an arbitrary test function $\varphi_h(x)$ in \eqref{pathwise Ritz},
multiply with arbitrary $w \in L^2(\Omega)$, utilise $w(\omega)\Ng \varphi_h(x) = \Ng (w(\omega) \varphi_h(x))$,
and integrate over $\Omega$ to obtain
\[
\int_\Omega \int_{\Gt} \alpha(\omega,x) \Ng (u(\omega,x) - u_p(\omega,x))\Ng (\varphi_h(x) w(\omega))= 0.
\]
Since $\{ v(x) w(\omega)\;|\; v \in S_h^l(t), w \in L^2(\Omega) \}$ 
is a dense subset of $V_h(t)$, 
the Galerkin orthogonality (\ref{proj res}) follows.
\end{proof}

An error estimate for the pathwise Ritz projection $u_p$ defined in \eqref{eq:PWRITZ} 
is established in the next theorem.

\begin{theorem}\label{error in projection}
For each fixed $t\in [0,T]$, the pathwise Ritz projection 
$u_p\in L^2(\Omega, S_h^l(t))$ of $u \in L^2(\Omega, H^2(\Gt))$
satisfies the error estimate
\begin{equation}
\| u - u_p \|_{H(t)} + h \| \Ng (u-u_p) \|_{H(t)} 
\leq ch^2 \| u \|_{L^2(\Omega, H^2(\Gt))} \label{6.2}
\end{equation}
with a constant $c$ depending only on 
the properties of $\alpha$ as stated in Assumptions~\ref{ass1} and~\ref{ass1} and
the shape regularity of $\Gamma_h(t)$.
\end{theorem}
\begin{proof} 
The Galerkin orthogonality (\ref{proj res}) and \eqref{uniform} provide
\begin{align*}
\alpha_{\min }\|\Ng(u-u_p)\|_{H(t)}
&\leq \alpha_{\max} \inf_{v\in L^2( \Omega, S_h^l(t))}\|\Ng(u-v)\|_{H(t)}\\
&\leq \alpha_{\max} \|\Ng(u-I_hv)\|_{H(t)}.
\end{align*}
Hence, the bound for the gradient follows directly from Lemma~\ref{int estimate}.

In order to get the second order bound, we will use a Aubin-Nitsche duality argument. 
For every fixed $\omega \in \Omega$, we consider the path-wise problem to 
find $w(\omega) \in H^1(\Gt)$ with $\int_{\Gt} w = 0$ such that
\begin{equation}\label{eq:ellipticweakform}
  \int_{\Gamma(t)} \alpha  \Ng w(\omega) \cdot \Ng \varphi 
  =  \int_{\Gamma(t)} (u - u_p)\varphi \qquad \forall \varphi \in H^1(\Gamma(t)).
\end{equation}
Since $\Gamma(t)$ is $\mathcal{C}^2$, it follows by \cite[Theorem 3.3]{DE13a} that $w(\omega) \in H^2(\Gt)$. Inserting the test function $\varphi = w(\omega)$ into (\ref{eq:ellipticweakform}) and  utilizing the Poincar\'{e}'s inequality, we obtain
\[
\| \nabla_\Gamma w(\omega) \|_{L^2(\Gamma(t))} \leq \frac{C_P}{\alpha_{\min}} \|u-u_p\|_{L^2(\Gamma(t))}.
\]
Previous estimate together with the product rule for the divergence imply
\[
\| \Delta_\Gamma w(\omega) \|_{L^2(\Gamma(t))} \leq \frac{1}{\alpha_{\min}} \|u-u_p\|_{L^2(\Gamma(t))} + \frac{C_P}{
\alpha_{\min}^2} \|\alpha(\omega)\|_{\mathcal{C}^1(\Gamma(t))} \| u - u_p \|_{L^2(\Gt)}.
\]
Hence, we have the following estimate
\begin{equation}\label{H2reg}
\|w(\omega)\|_{H^2(\Gt)} \leq C \| u - u_p \|_{L^2(\Gt)},
\end{equation}
with a constant $C$ depending only on 
the properties of $\alpha$ as stated in Assumptions~\ref{ass1} and~\ref{ass2}.
Furthermore, well-known results on random elliptic pdes
with uniformly bounded coefficients~\cite{Celliptic,CST} imply measurablility of $w(\omega)$, $\omega \in \Omega$.
Integrating \eqref{H2reg} over $\Omega$, we therefore obtain 
\begin{equation} \label{eq:WEST}
\| w \|_{L^2(\Omega, H^2(\Gt))} \leq C \| u - u_p \|_{H(t)}.
\end{equation}
Using again Lemma~\ref{int estimate}, Galerkin orthogonality (\ref{proj res}), and \eqref{eq:WEST}, we get
\begin{align*}
\| u - u_p \|^2_{H(t)} &= a(w, u - u_p) \
=a(w-I_hw,u-u_p) \\
&\leq \alpha_{\max} \|\Ng(w-I_hw) \|_{H(t)} \| \Ng(u-u_p)\|_{H(t)} \\
&\leq c'h^2 \| w \|_{L^2(\Omega, H^2(\Gt))}  \| u \|_{L^2(\Omega, H^2(\Gt))}\\
&\leq c'c h^2 \| u-u_p\|_{H(t)} \| u \|_{L^2(\Omega, H^2(\Gt))}.
\end{align*}
with a constant $c'$ depending on the shape regularity of $\Gamma_h(t)$.
This completes the proof.
\end{proof}
\begin{xrem}
The first order error bound for $\| \Ng (u-u_p) \|_{H(t)}$ still holds,
if spatial regularity of $\alpha$ as stated in Assumption~\ref{ass2} is not satisfied.
\end{xrem}
We conclude with an error estimate for the material derivative of $u_p$
that can be proved as in the deterministic setting \cite[Theorem 6.2 ]{DE13}.

\begin{theorem} \label{thm:MATDERRI}
For each fixed $t\in [0,T]$, the discrete material derivative of the pathwise Ritz projection satisfies the error estimate
\begin{equation}\label{error md Rproj}
\begin{split}
\| \mdh u - \mdh u_p \|_{H(t)}+ h \| \Ng(\mdh u - \mdh u_p) \|_{H(t)} \\
\leq ch^2(\|u\|_{L^2(\Omega,H^2(\Gamma))} + \|\md u\|_{L^2(\Omega,H^2(\Gamma))})
\end{split}
\end{equation}
with a constant $C$ depending only on 
the properties of $\alpha$ as stated in Assumptions~\ref{ass1} and~\ref{ass2}.
\end{theorem}

\subsection{Error estimates for the evolving surface finite element discretization}
Now we are in the position to state an error estimate
for the evolving surface finite element discretization of Problem~\ref{prob:WF2} 
as formulated in  Problem~\ref{prob:ESFEM}.

\begin{theorem}\label{L2error}
Assume that the solution $u$ of Problem~\ref{prob:WF2} has the regularity properties
\begin{equation}
\sup_{t \in (0,T)}\| u(t)\|_{L^2(\Omega, H^2(\Gt))} +
\int_0^T 
\| \md u (t) \|^2_{L^2(\Omega, H^2(\Gt))} dt < \infty \label{4.26}
\end{equation} 
and let $U_h \in W_h(V_h,H_h)$ be the solution of the approximating Problem~\ref{prob:ESFEM}
with an initial condition $U_h(0)=U_{h,0}\in V_h(0)$ such that 
\begin{equation}
\| u(0) - U_{h,0}^l \|_{H(0)} \leq c h^2 \label{4.27}
\end{equation}
holds with a constant $c>0$ independent of $h$. Then the lift $u_h:=U_h^l$
satisfies the error estimate
\begin{equation}
\sup_{t \in (0,T)} \| u(t) - u_h(t)\|_{H(t)} \leq Ch^2 \label{l2 error estimate}
\end{equation}
with a constant $C$ independent of $h$.
\end{theorem}

\begin{proof}
Utilizing the preparatory results from the preceding sections, 
the proof can be carried out in analogy to the deterministic version stated in~\cite[Theorem 4.4]{DE13}.

The first step is to decompose the error for fixed $t$ into the pathwise Ritz projection error and 
the deviation of the pathwise Ritz projection $u_p$ from the approximate solution $u_h$ according to
\begin{equation*}
\|u(t) - u_h(t)\|_{H(t)} \leq \|u(t) - u_p(t)\|_{H(t)} + \|u_p(t) - u_h(t)\|_{H(t)}, \quad t \in (0,T).
\end{equation*}
For ease of presentation the dependence on $t$ is often skipped in the sequel.

As a consequence of Theorem \ref{error in projection} and the regularity assumption \eqref{4.26}, we have
\begin{equation*}
\sup_{t \in (0,T)} \|u - u_p\|_{H(t)} \leq c h^2 \sup_{t \in (0,T)}\| u \|_{L^2(\Omega, H^2(\Gt))} < \infty.
\end{equation*}
Hence, it is sufficient to show a corresponding estimate for 
\begin{equation*}
\theta := u_p - u_h \in L^2(\Omega, S_h^l).
\end{equation*}
Here and in the sequel we set  $\vfi_h=\phi^l_h$ for $\phi_h \in L^2(\Omega, S_h)$.

Utilizing \eqref{semidiscrete problem} and the transport formulae \eqref{eq:TPF1} in Lemma~\ref{lem:PT} 
and \eqref{eq:TPF2} in  Lemma~\ref{lem:TPF2}, respectively, we obtain
\begin{equation}
\frac{d}{dt}m(u_h, \vfi_h) + a(u_h, \vfi_h) - m(u_h, \mdh\vfi_h) = F_1(\vfi_h), 
\quad \forall \vfi_h \in L^2(\Omega, S_h^l) \label{7.1}
\end{equation}
denoting 
\begin{align}
  \nonumber
F_1(\vfi_h) &:= m(\mdh u_h,\vfi_h)-m_h(\mdh U_h,\phi_h) \\
& \qquad +a(u_h, \vfi_h) - a_h(U_h, \phi_h)+ g(v_h; u_h, \vfi_h) - g_h(V_h; U_h, \phi_h). \label{eq:F1FIRST} 
\end{align}
Exploiting that $u$ solves  Problem~\ref{prob:WF2} and thus satisfies \eqref{3.8}
together with the Galerkin orthogonality \eqref{proj res}
and rearranging terms, we derive
\begin{equation}
\frac{d}{dt}m(u_p, \vfi_h) + a(u_p, \vfi_h) - m(u_p, \mdh\vfi_h) = F_2(\vfi_h), \quad \forall \vfi_h \in L^2(\Omega, S_h^l) \label{7.3}
\end{equation}
denoting
\begin{equation}
F_2(\vfi_h) := m(u,\md \vfi_h- \mdh\vfi_h)+ m(u-u_p, \mdh\vfi_h)-\frac{d}{dt}m(u-u_p, \vfi_h). \label{7.4}
\end{equation}

We subtract \eqref{7.1} from \eqref{7.3} to get
\begin{equation}\label{theta eq}
\frac{d}{dt}m(\theta, \vfi_h) + a(\theta, \vfi_h) - m(\theta, \mdh \vfi_h) 
= F_2(\vfi_h) - F_1(\vfi_h) \quad \forall \vfi_h \in L^2(\Omega, S_h^l). 
\end{equation}
Inserting the test function $\varphi_h = \theta\in  L^2(\Omega, S_h^l)$ into (\ref{theta eq}),
utilizing the transport Lemma~\ref{lem:TPF2}, 
and integrating  in time, we obtain
\begin{equation*}
{\textstyle \frac{1}{2} } \| \theta(t) \|^2_{H(t)} - {\textstyle \frac{1}{2}} \| \theta(0) \|^2_{H(0)}+
\intt  a(\theta, \theta) + \intt  g(\w_h; \theta,\theta) = \int_0^t F_2(\theta) - F_1(\theta).
\end{equation*}
Hence,  Assumption~\ref{ass1} 
together with \eqref{eq:GBOUND} in Remark \ref{rm: divh bound} provides the estimate
\begin{equation}\label{mtheta}
\begin{array}{rl}
 {\textstyle \frac{1}{2} } \displaystyle \|\theta(t)\|^2 + \alpha_{\min} \intt \| \Ng \theta\|^2_{H(t)} \leq & \\
  {\textstyle \frac{1}{2} } \|\theta(0)\|^2 +  & \displaystyle c \intt  \|\theta\|^2_{H(t)} + \intt  |F_1(\theta)| + |F_2(\theta)|.
  \end{array}
\end{equation}
Lemma \ref{bilinear errors} allows to control the geometric error terms in $|F_1(\theta)|$ according to
\begin{equation*} 
 |F_1(\theta)| \leq ch^2\|\mdh u_h \|_{H(t)} \| \theta_h \|_{H(t)} +ch^2\| u_h \|_{V(t)} \| \theta_h \|_{V(t)}.
\end{equation*}
The transport formula~\eqref{eq:TPF2} provides the identity
\[
F_2(\vfi_h) = m(u,\md \vfi_h- \mdh\vfi_h) - m(\mdh(u-u_p),\vfi_h) - g(v_h;u-u_p,\vfi_h)
\]
from which Lemma \ref{l5.6}, Theorem~\ref{thm:MATDERRI}, and Theorem~\ref{error in projection} imply
\begin{equation*}
|F_2(\theta)| \leq  ch^2\|u\|_{H(t)}\|\theta_h\|_{V(t)}
+ ch^2(\|u\|_{L^2(\Omega, H^2(\Gamma(t)))} +\|\md u\|_{L^2(\Omega, H^2(\Gamma(t)))})\|\theta_h\|_{H(t)}. 
\end{equation*}
We insert these estimates into \eqref{mtheta}, rearrange terms,
and apply Young's inequality to show
that for each $\varepsilon > 0$ there is a positive constant $c(\varepsilon)$ such that 
\begin{multline*}
\frac{1}{2}\| \theta(t) \|^2_{H(t)} + (\alpha_{\min}-\varepsilon) \intt \|\Ng \theta\|^2_{H(t)}
\leq
\frac{1}{2}\| \theta(0) \|^2_{H(0)}+ c(\varepsilon) \intt  \| \theta \|^2_{H(t)} \\
+ c(\varepsilon) h^4\intt \left( \|u\|^2_{L^2(\Omega, H^2(\Gamma(t)))}
+ \|\md u\|^2_{L^2(\Omega, H^2(\Gamma(t)))} + \| \mdh u \|^2_{H(t)} + \| u_h \|^2_{V(t)} \right).
\end{multline*}
For sufficiently small $\varepsilon>0$, Gronwall's lemma implies
\begin{equation}
\sup_{t \in (0,T)} \| \theta(t) \|^2_{H(t)} + \int_0^T \| \Ng \theta \|^2_{H(t)} \leq c \|\theta(0)\|^2_{H(0)} + ch^4C_h, \label{7.8}
\end{equation}
where
\begin{equation*}
C_h = \int_0^T [\|u\|^2_{L^2(\Omega, H^2(\Gamma(t))} +
\|\md u\|^2_{L^2(\Omega, H^2(\Gamma(t))} + \| \mdh u \|^2_{H(t)} + \| u_h \|^2_{V(t)}].
\end{equation*}
Now the consistency assumption \eqref{4.27} yields $\|\theta(0)\|^2_{H(0)} \leq c h^4$ while
the stability result \eqref{apriorilifting} in Remark~\ref{rem:STABLIFT} together with
the regularity assumption leads to \eqref{4.26} $C_h \leq C <\infty$
with a constant $C$ independent of $h$.
This completes the proof.
\end{proof}

\begin{xrem}
Observe that without Assumption \ref{ass2} we still get the $H^1$-bound
\[
\left(\int_0^T \| \Ng(u(t) - u_h(t)) \|^2_{H(t)} \right)^{1/2}\leq C h.
\]
\end{xrem}

The following error estimate for the expectation 
\[
E[u] =\int_{\Omega}u 
\]
is an immediate consequence of Theorem~\ref{L2error} and the Cauchy-Schwarz inequality.
\begin{theorem} \label{theo:EXERR}
Under the assumptions and with the notation of Theorem~\ref{L2error} we have the error estimate
\begin{equation} \label{eq:EXPERR}
 \sup_{t\in (0,T)}\|E [u(t)] - E [u_h(t)]\|_{L^2(\Gamma(t))}\leq Ch^2.
\end{equation}
\end{theorem}

We close this section with an error estimate for the Monte-Carlo approximation of the expectation $E [u_h]$.
Note that $E[u_h](t)=E[u_h(t)]$, because the probability measure does not depend on time $t$.
For each fixed $t\in (0,T)$ and some $M \in \mathbb{N}$, 
the Monte-Carlo approximation $E_M [u_h](t)$ of  $E [u_h](t)$ is defined by 
\begin{equation}\label{MCerror}
 E_M[u_h(t)]:=\frac{1}{M} \sum_{i=1}^{M} u_h^i(t) \in  L^2(\Omega^M,L^2(\Gamma(t))),
\end{equation}
where $u_h^i$ are 
independent identically distributed copies of the random field $u_h$.

A  proof of the following well-known result can be found, e.g. in \cite[Theorem 9.22]{LPS}.
\begin{lemma} \label{lem:MC}
For each fixed $t\in (0,T)$, $w \in L^2(\Omega,L^2(\Gamma(t)))$, and any $M \in \mathbb{N}$
we have the error estimate
\begin{equation} \label{MC}
\textstyle \| E[w] - E_M[w] \|_{L^2(\Omega^M,L^2(\Gamma(t)))} 
= \frac{1}{\sqrt{M}} 
{\mathrm Var}[w]^{\frac{1}{2}} 
\leq \frac{1}{\sqrt{M}} \| w \|_{L^2(\Omega,L^2(\Gamma(t)))}
\end{equation}
with ${\mathrm Var}[w]$ denoting the variance $Var[w] = E[\| E[w] - w \|_{L^2(\Omega,\Gamma(t))}^2]$ of $w$.
\end{lemma}

\begin{theorem} \label{theo:MCESFEM}
Under the assumptions and with the notation of Theorem~\ref{L2error} we have the error estimate
\[
  \sup_{t \in (0,T)} \|E[u](t)-E_M[u_h](t) \|_{L^2(\Omega^M,L^2(\Gamma(t)))} 
  \leq C\left( h^2 + \textstyle \frac{1}{\sqrt{M}}\right)
\]
with a constant $C$ independent of $h$ and $M$.
\end{theorem}

\begin{proof}
Let us first note that
\begin{equation} \label{eq:SUP}
 \sup_{t\in (0,T)}\|u_h\|_{H(t)}\leq (1+C)\sup_{t\in (0,T)} \|u\|_{H(t)}< \infty
\end{equation}
follows from the triangle inequality and Theorem~\ref{L2error}.
For arbitrary fixed $t\in (0,T)$ the triangle inequality yields
\[
\begin{array}{rl}
\|E[u](t)-E_M[u_h](t) \|_{L^2(\Omega^M,L^2(\Gamma(t)))} \leq &\\[3mm]
\|E[u](t)-E[u_h](t) \|_{L^2(\Gamma(t)))} \;\; + &
\|E[u_h(t)]-E_M[u_h(t)] \|_{L^2(\Omega^M,L^2(\Gamma(t)))}
\end{array}
\]
so that the assertion follows from Theorem~\ref{theo:EXERR}, Lemma~\ref{lem:MC}, and \eqref{eq:SUP}.
\end{proof}

\section{Numerical Experiments}
\subsection{Computational aspects} \label{sec:COMPAS}
In the following numerical computations we consider a fully discrete scheme
as resulting from an implicit Euler discretization of the semi-discrete Problem~\ref{prob:ESFEM}.
More precisely, we select a time step $\tau>0$ with $K\tau =T$, set 
\[
\chi_j^k=\chi_j(t_k),\quad k=0,\dots, K,
\]
with $t_k=k\tau$, and approximate $U_h(\omega,t_k)$ by 
\[
U_h^k(\omega)=\sum_{j=1}^J U_j^k(\omega) \chi_j^k,\qquad k=0,\dots,J,
\]
with unknown coefficients $U_j^k(\omega)$ characterized by the initial condition
\[
 U_h^0=\sum_{j=1}^J  U_{h,0}(X_j(0)) \chi_j^0
\]
and the fully discrete scheme 
\begin{equation}\label{eq:DISCPRO}
\frac{1}{\tau}\big(m_h^k(U_h^{k},\chi_j^k) -m_h^{k-1}(U_h^{k-1}, \chi_j^{k-1}) \big) + a_h^k(U_h^k,\chi_j^k) 
=  \int_{\Omega}\int_{\Gamma(t_k)} f(t_k) \chi_j^k
\end{equation}
for $k=1,\dots,J$. Here, for $t=t_k$ the time-dependent bilinear forms $m_h(\cdot,\cdot)$ and $a_h(\cdot,\cdot)$
are denoted by $m_h^k(\cdot,\cdot)$ and $a_h^k(\cdot,\cdot)$, respectively.
The fully discrete scheme \eqref{eq:DISCPRO} is obtained 
from an extension of \eqref{semidiscrete problem} to non-vanishing 
right-hand sides $f\in {\mathcal C}((0,T),H(t))$  by inserting $\varphi = \chi_j$, 
exploiting \eqref{tr prop}, and replacing the time derivative by the backward difference quotient.
As  $\alpha$ is defined on the whole ambient space in the subsequent numerical experiments,
the inverse lift $\alpha^{-l}$ occurring in $a_h(\cdot,\cdot)$ is replaced by $ \alpha|_{\Gamma_h(t)}$, and the integral is computed using a quadrature formula of degree 4.

The expectation $E[U_h^k]$ is approximated by the Monte-Carlo method
\[
E_M[U_h^k]=\frac{1}{M}\sum_{i=1}^M U_h^k(\omega^i), \qquad k=1,\dots,K,
\]
with independent, uniformly distributed samples $\omega^i \in \Omega$.
For each sample $\omega^i$, the evaluation of $U_h^k(\omega^i)$ from
the initial condition and \eqref{eq:DISCPRO}
amounts to the solution of $J$ linear systems which is performed by 
iteratively by a preconditioned conjugate gradient method up to the accuracy $10^{-8}$.

From our theoretical findings stated in Theorem~\ref{theo:MCESFEM} 
and the fully discrete deterministic results in \cite[Theorem~2.4]{DziEll12}, 
we expect that the discretization error
\begin{equation} \label{eq:DREAM}
 \sup_{k=0,\dots,K} \| E[u](t_k) - E_M[U_h^k] \|_{L^2(\Omega^M,L^2(\Gamma_h(t_k)))}
\end{equation}
behaves like ${\mathcal O}\left(h^2 + \textstyle \frac{1}{\sqrt{M}} + \tau\right)$. 
This conjecture will be investigated in our numerical experiments.
To this end, the integral over $\Omega^M$ in \eqref{eq:DREAM} is always approximated by 
the average of 20 independent, identically distributed sets of samples.

The implementation was carried out in the framework of {\sc Dune} 
(Distributed Unified Numerics Environment)~\cite{BasBlaDed08-a,BasBlaDed08-b,DedKloNol10},
and the corresponding code is available at \url{https://github.com/tranner/dune-mcesfem}.

\subsection{Moving curve}

We consider the ellipse 
\begin{align*}
  \Gamma(t) = \left\{ x=(x_1,x_2)\in \R^2\;\middle | \;  \frac{x_1^2}{a(t)} + \frac{x_2^2}{b(t)}= 1 \right\},
 \qquad t \in [0,T],
\end{align*}
with oscillating axes $a(t) = 1 + \frac{1}{4} \sin(t)$, $b(t) = 1 + \frac{1}{4} \cos(t)$, and $T=1$.
The random diffusion coefficient $\alpha$ occurring in $a_h(\cdot,\cdot)$
is given by
\begin{align*}
  \alpha( x, \omega ) = 1 + \frac{Y_1(\omega)}{4} \sin( 2 x_1 ) + \frac{Y_2(\omega)}{4} \sin( 2 x_2 ),
\end{align*}
where  $Y_1$ and $Y_2$
stand for independent, uniformly distributed random variables on $\Omega=(-1,1)$.
The  right-hand side $f$ in \eqref{eq:DISCPRO} is selected in such a way that for each $\omega \in \Omega$ 
the exact solution of the resulting path-wise problem is given by
\begin{align*}
  u( x, t, \omega ) 
  = \sin(t) \big\{ \cos(3x_1) + \cos(3x_2) + Y_1(\omega) \cos(5 x_1) + Y_2(\omega) \cos( 5 x_2 ) \big\},
\end{align*}
and we set $u_0(x,\omega) = u( x, 0, \omega ) = 0$.

The initial polygonal approximation $\Gamma_{h,0}$ of $\Gamma(0)$ is depicted in Figure~\ref{fig:GRID1}
for the mesh sizes $h=h_j$, $j=0,\dots,4$, that are used in our computations.
\begin{figure}[htb] \label{fig:GRID1}
  \centering
  \includegraphics[width=0.20\textwidth]{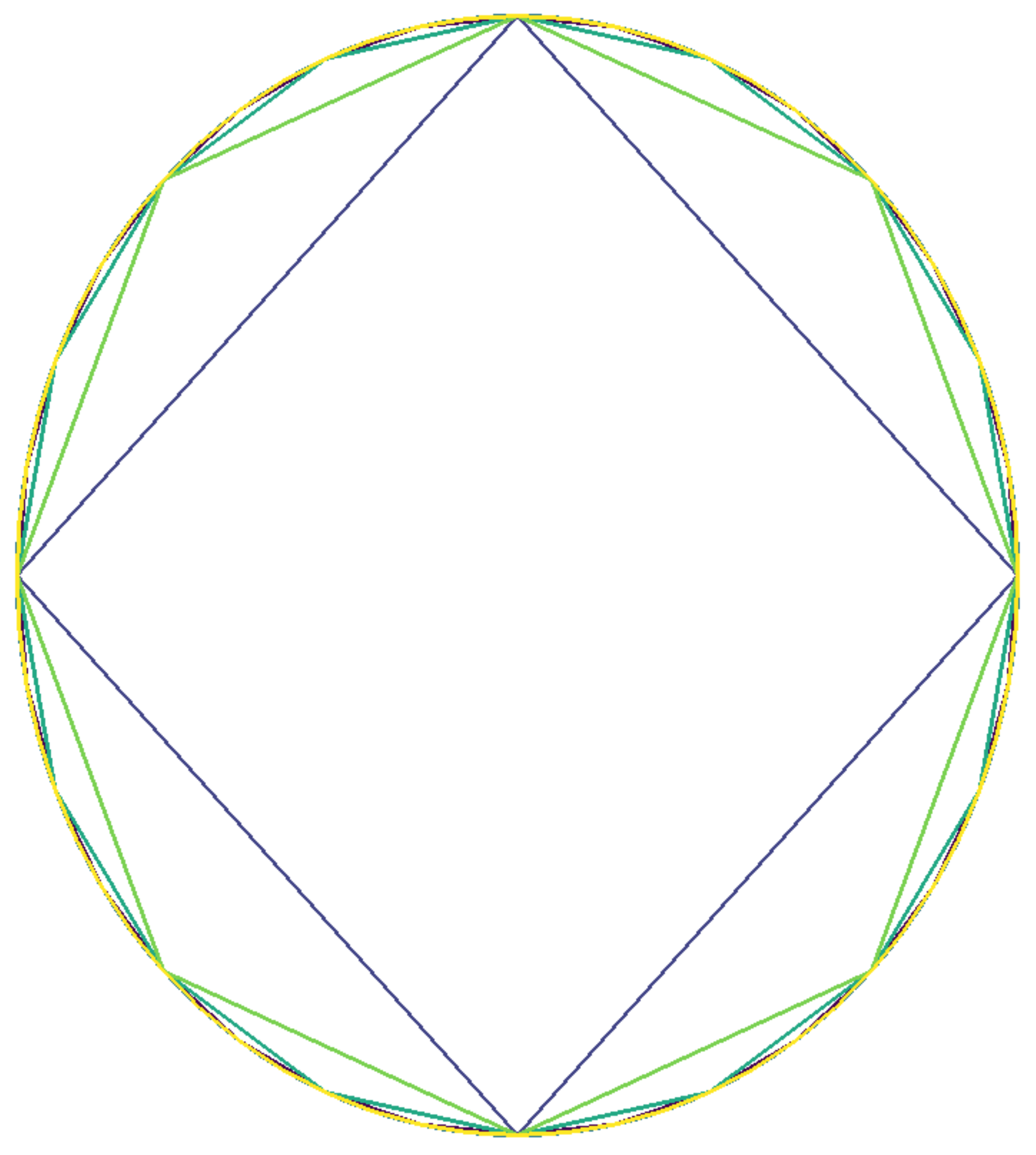}
  \caption{Polygonal approximation $\Gamma_{h,0}$ of $\Gamma(0)$ for $h=h_0,\dots, h_4$.}
\end{figure}
We select the corresponding time step sizes
$\tau_0=1$, $\tau_j=\tau_{j-1}/4$
and the corresponding numbers of samples $M_1=1$, $M_j=16 M_{j-1}$ for $j=1,\dots,4$. 
The resulting discretization errors \eqref{eq:DREAM} are
reported in Table~\ref{tab:2d-evolve-opt-l2-results} suggesting
the optimal behavior ${\mathcal O}(h^2 + M^{-1/2} + \tau)$.

\begin{table}[htb]\label{tab:2d-evolve-opt-l2-results}
  \centering
  \begin{tabular}{ccc|c|ccc}
    $h$ & $M$ & $\tau$ & Error & eoc($h$) & eoc($M$) & eoc($\tau$)\\
   \hline
$1.500000$ & $1$ & $1$ & $3.00350$ & --- & --- & --- \\
$0.843310$ & $16$ & $4^{ -1 }$ & $2.23278 \cdot 10^{-1}$ & $4.51325$ & $-0.93743$ & $1.87487$ \\
$0.434572$ & $256$ & $4^{ -2 }$ & $1.86602 \cdot 10^{-1}$ & $0.27066$ & $-0.06472$ & $0.12944$ \\
$0.218962$ & $4\,096$ & $4^{ -3 }$ & $4.88096 \cdot 10^{-2}$ & $1.95642$ & $-0.48368$ & $0.96736$ \\
$0.109692$ & $65\,536$ & $4^{ -4 }$ & $1.29667 \cdot 10^{-2}$ & $1.91768$ & $-0.47809$ & $0.95618$ \\
  \end{tabular}
  \caption{Discretization errors for a moving curve in $\R^2$.}
\end{table}

\subsection{Moving surface}
We consider the ellipsoid
\begin{align*}
   \Gamma(t) = \left\{ x=(x_1,x_2,x_3)\in \R^3\;\middle | \; \frac{x_1^2}{a(t)} + x_2^2 + x_3^2 = 1  \right\},
 \qquad t \in [0,T],
\end{align*}
with oscillating $x_1$-axis $a(t) = 1 + \tfrac{1}{4} \sin(t)$ and $T=1$.
The random diffusion coefficient $\alpha$ occurring in $a_h(\cdot,\cdot)$
is given by
\begin{align*}
  \alpha( x, \omega ) = 1 + x_1^2 + Y_1( \omega ) x_1^4 + Y_2( \omega ) x_2^4,
\end{align*}
where  $Y_1$ and $Y_2$
denote independent, uniformly distributed random variables on $\Omega=(-1,1)$.
The  right-hand side $f$ in \eqref{eq:DISCPRO} is chosen such that for each $\omega \in \Omega$ 
the exact solution of the resulting path-wise problem is given by
\begin{align*}
  u( x, t, \omega ) =  \sin( t ) x_1 x_2 + Y_1( \omega) \sin(2t) x_1^2 + Y_2( \omega ) \sin(2t) x_2,
\end{align*}
and we set $u_0(x,\omega) = u( x, 0, \omega ) =  0$.

The initial triangular approximation $\Gamma_{h,0}$ of $\Gamma(0)$ is depicted in Figure~\ref{fig:GRID2}
for the mesh sizes $h=h_j$, $j=0,\dots,3$.
\begin{figure}[htb] \label{fig:GRID2}
  \centering
  \includegraphics[width=0.22\textwidth]{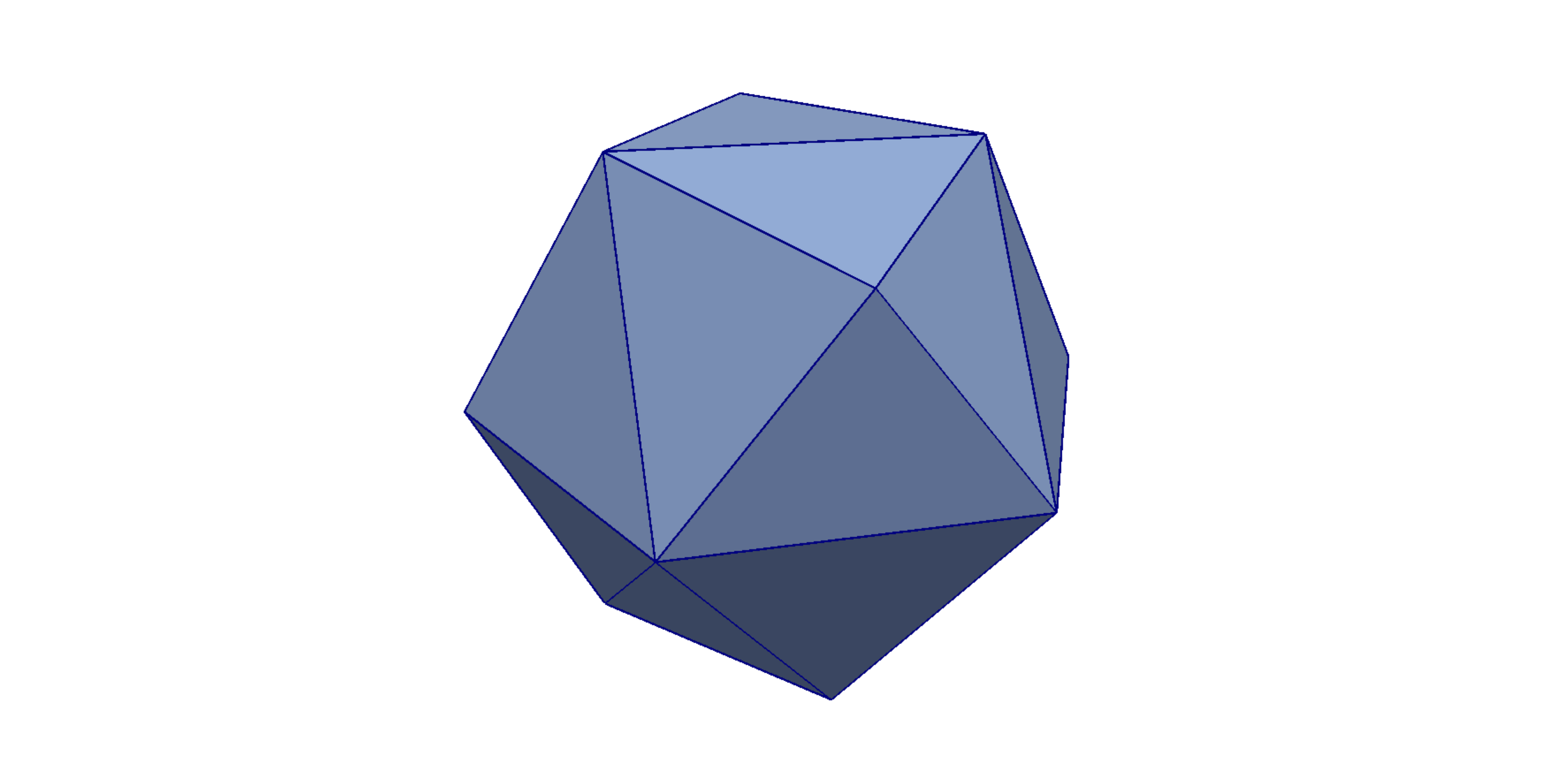}
  \includegraphics[width=0.22\textwidth]{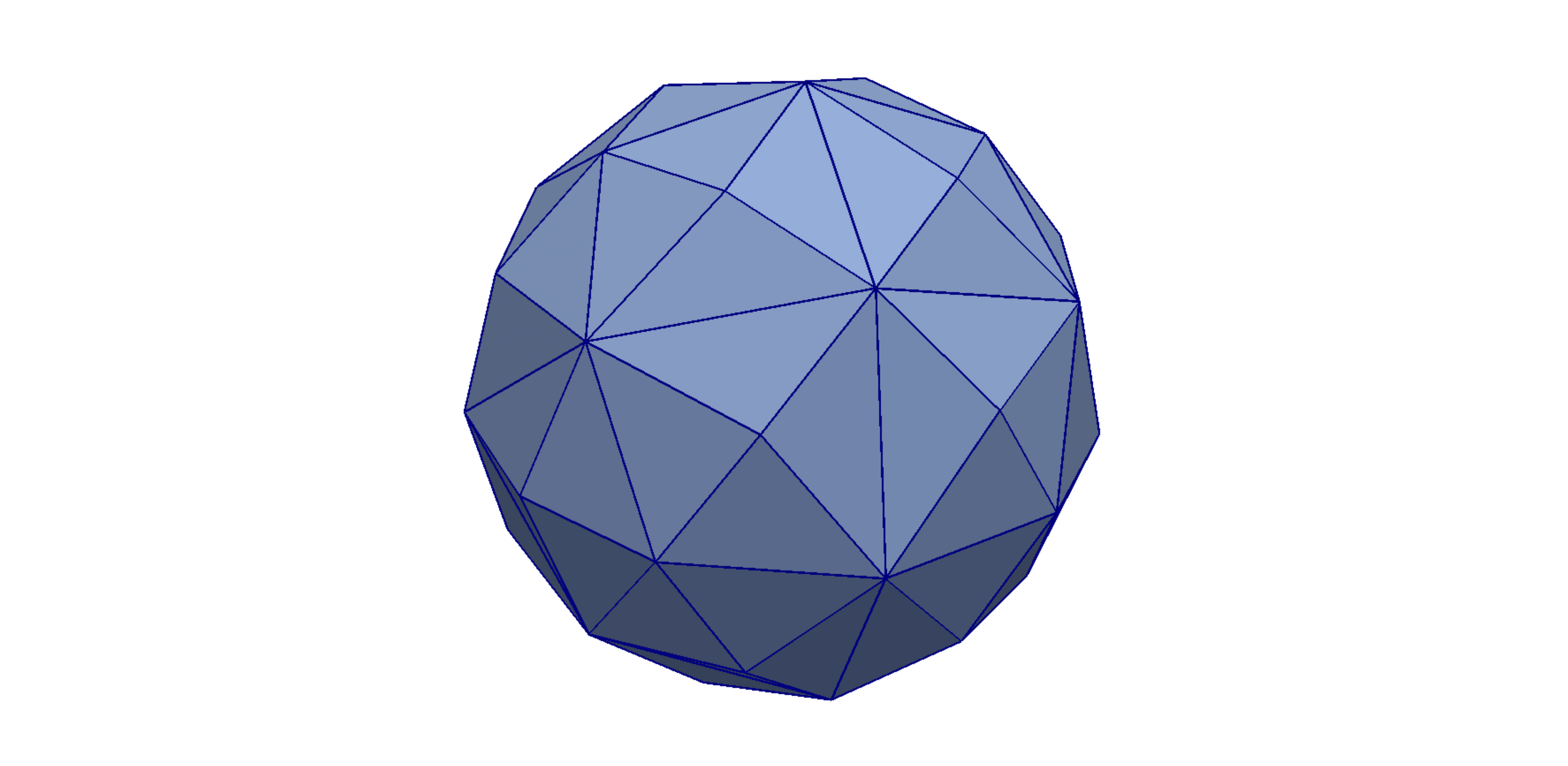}
  \includegraphics[width=0.22\textwidth]{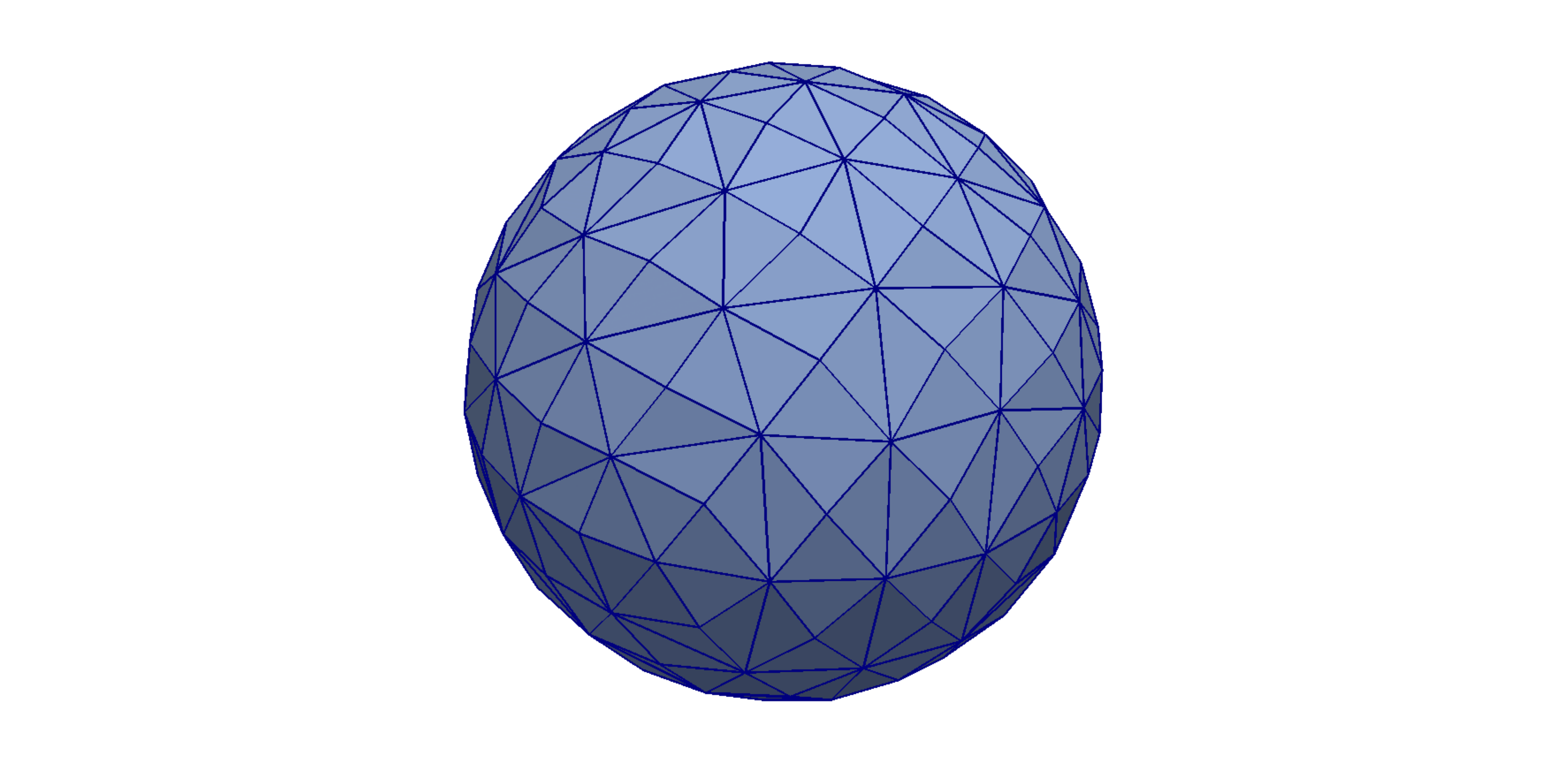}
  \includegraphics[width=0.22\textwidth]{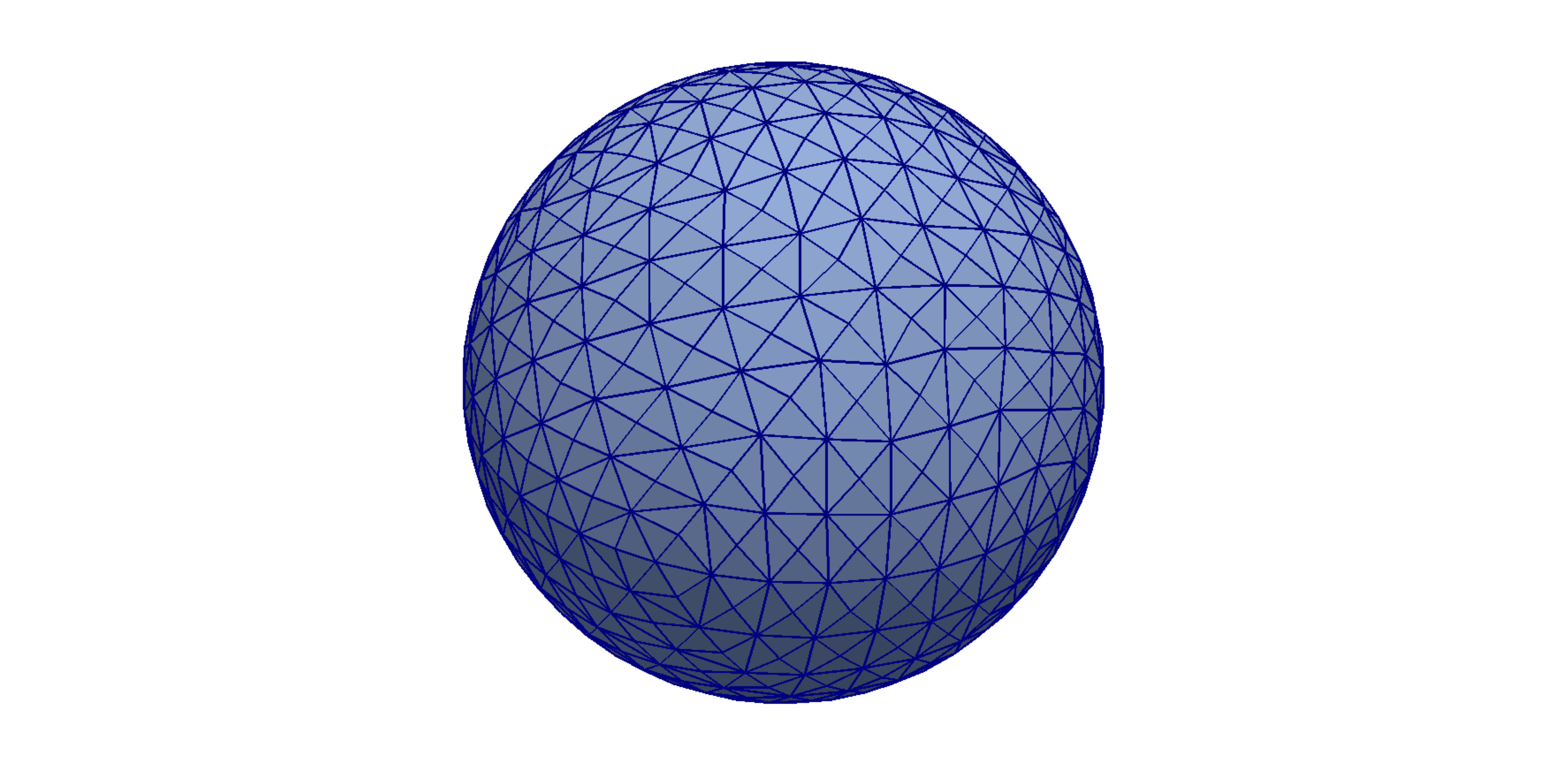}
    \caption{Triangular approximation $\Gamma_{h,0}$ of $\Gamma(0)$ for $h=h_0,\dots, h_3$.}
\end{figure}
\begin{table}[htb]
  \centering
  \begin{tabular}{ccc|c|ccc}
    $h$ & $M$ & $\tau$ & Error & eoc($h$) & eoc($M$) &  eoc($\tau$) \\
    \hline
$1.276870$ & $1$ & $1$ & $9.91189 \cdot 10^{-1}$ & --- & --- & --- \\
$0.831246$ & $16$ & $4^{ -1 }$ & $1.70339 \cdot 10^{-1}$ & $4.10285$ & $-0.63519$ & $1.27037$ \\
$0.440169$ & $256$ & $4^{ -2 }$ & $4.61829 \cdot 10^{-2}$ & $2.05293$ & $-0.47075$ & $0.94149$ \\
$0.222895$ & $4\,096$ & $4^{ -3 }$ & $1.18779 \cdot 10^{-2}$ & $1.99561$ & $-0.48977$ & $0.97954$ \\
  \end{tabular}
  \caption{Discretization errors for a moving surface in $\R^3$.}
  \label{tab:3d-opt-l2-results}
\end{table}
We select the corresponding time step sizes
$\tau_0=1$, $\tau_j=\tau_{j-1}/4$
and the corresponding numbers of samples $M_1=1$, $M_j=16 M_{j-1}$ for $j=1,2,3$. 
The resulting discretization errors \eqref{eq:DREAM} are
shown in Table~\ref{tab:3d-opt-l2-results}.
Again, we observe that the discretization error  behaves like ${\mathcal O}(h^2 + M^{-1/2} + \tau)$.
This is in accordance with our
theoretical findings stated in Theorem~\ref{theo:MCESFEM} 
and fully discrete deterministic results \cite[Theorem~2.4]{DziEll12}.

\bibliographystyle{plainnat}
\bibliography{refs.bib}

\end{document}